\documentclass[11pt,a4paper]{article}

\usepackage{amsfonts,amsmath,wasysym}

\usepackage{tikz}
\usetikzlibrary{shapes,positioning,patterns}
\usepackage{graphicx} 

\usepackage{amssymb,amsthm,paralist}
\usepackage{latexsym}
\usepackage{url}

\definecolor{darkgreen}{rgb}{0,0.5,0}
\definecolor{darkred}{rgb}{0.7,0,0}
\usepackage[colorlinks, citecolor=darkgreen, linkcolor=darkred
]{hyperref}

\theoremstyle{plain}

\numberwithin{equation}{section}

\newcommand{\h}{\ensuremath{{\mathcal H}}}

\newcommand{\w}{\ensuremath{{\mathcal W}}}

\newcommand{\ca}{\ensuremath{{\mathcal A}}}

\newcommand{\cl}{\ensuremath{{\mathcal L}}}

\newcommand{\cv}{\ensuremath{{\mathcal V}}}

\newcommand{\pl}[2]{{\frac{\partial #1}{\partial #2}}}

\newcommand{\ti}{\tilde}
\newcommand{\al}{\alpha}
\newcommand{\be}{\beta}
\newcommand{\ga}{\gamma}

\newcommand{\de}{\delta}

\newcommand{\ka}{\kappa}

\newcommand{\si}{\sigma}

\renewcommand{\th}{\theta}

\newcommand{\vph}{\varphi}
\newcommand{\ep}{\varepsilon}

\newcommand{\R}{\ensuremath{{\mathbb R}}}
\newcommand{\N}{\ensuremath{{\mathbb N}}}

\newcommand{\downto}{\downarrow}
\newcommand{\upto}{\uparrow}

\newcommand{\intersect}{\cap}

\newcommand{\beq}{\begin{equation}}
\newcommand{\beql}[1]{\begin{equation}\label{#1}}
\newcommand{\eeq}{\end{equation}}
\newcommand{\beqa}{\begin{equation}\begin{aligned}}
\newcommand{\eeqa}{\end{aligned}\end{equation}}
\newcommand{\brmk}{\begin{rmk}}
\newcommand{\ermk}{\end{rmk}}
\newcommand{\partref}[1]{\hbox{(\csname @roman\endcsname{\ref{#1}})}}
\newcommand{\half}{\frac{1}{2}}

\newtheorem{thm}{Theorem}[section]
\newtheorem{cor}[thm]{Corollary}
\newtheorem{lem}[thm]{Lemma}

\newtheorem{rmk}[thm]{Remark}

\title{\sc delayed parabolic regularity for \\ curve shortening flow\footnote{A video about some of the content of this paper is hosted by \href{https://www.slmath.org/workshops/28375/schedules/36632}{SLMath}}}
\author{Arjun Sobnack and Peter M. Topping}
\date{25 September 2026}

\begin{document}

%

\parskip 8pt
\parindent 0pt

\maketitle

\begin{abstract}
Given two curves bounding a region of area $A$ that evolve under curve shortening flow, we propose the principle that the regularity of one should be controllable in terms of the regularity of the other, starting from time $A/\pi$. We prove several results of this form and demonstrate that no estimate can hold before that time.
As an example application, we construct solutions to graphical curve shortening flow starting with initial data that is merely an $L^1$ function. 
\end{abstract}


\section{Introduction}
\label{intro}

A smooth embedded curve in the plane is said to evolve under curve shortening flow if we deform it in time with velocity given by its geodesic curvature vector. 
In the case that the curve starts as a smooth embedded closed loop, 
the work of Gage, Hamilton and Grayson \cite{gage_hamilton, grayson} tells us that 
there is a unique smooth evolution under this flow, which lasts for 
a time $T=\frac{A}{2\pi}$, where $A$ is the initial enclosed area. 

In the case that the curve starts as the graph of a smooth function $y_0:\R\to\R$, 
the pioneering work of Ecker-Huisken \cite{EH1, EH2}, valid also for higher-dimensional mean curvature flow, tells us that there always exists a solution to the curve shortening flow irrespective of how badly behaved the function $y_0$ is at spatial infinity. A natural generalisation was given by Chou-Zhu \cite{CZ}, who ran the flow starting with an arbitrary proper embedded curve in the plane. 

Ecker-Huisken's flow can be represented as the graph of an evolving function
$y:\R\times [0,\infty)\to\R$ that solves the equation
\beq
\label{GCSF}
y_t=\frac{y_{xx}}{1+y_x^2}\equiv (\arctan(y_x))_x.
\eeq
Such a solution is called a graphical curve shortening flow, and uniqueness is known within this class by the work of Chou-Zhu \cite{CZ}.

A fundamental question when studying a geometric flow is to determine its regularising properties. This could be phrased in terms of decay estimates of some quantity that measures regularity, or by expressing a class of rough initial data from which we can start the flow. 
Ecker-Huisken's result \cite{EH2} could handle initial graphs $y_0$ that are merely \textit{locally} Lipschitz, yielding smooth 
solutions $y:\R\times (0,\infty)\to\R$
that achieve the initial data locally uniformly as $t\downto 0$.
The graphical curve shortening flow equation \eqref{GCSF} implies that the gradient 
$\al:=y_x$ evolves under the parabolic PDE
\beq
\label{alpha_eq}
\al_t=(\arctan \al)_{xx},
\eeq
so a bound on $\al$ induced by a Lipschitz hypothesis on $y$ will bring classical parabolic regularity theory into play.
However, Ecker-Huisken needed to prove local estimates to handle the fact that the initial curve is only \textit{locally} Lipschitz; these estimates ultimately inspired Perelman's pseudolocality theorem for Ricci flow \cite[\S 10]{P1}. The later work of Chou-Zhu \cite{CZ} also works at this level of regularity and includes the uniqueness of graphical solutions when the locally Lipschitz initial data is attained locally uniformly by the flow as $t\downto 0$; see also the uniqueness results of Chou-Kwong \cite{CK} and 
Daskalopoulos-Saez \cite{DS}.

More generally, Evans-Spruck \cite{ES3} considered the problem of starting the graphical flow (in arbitrary dimension) with the graph of a \textit{continuous} function. 
The essential regularity estimate they prove is that a local $C^0$ bound on the initial graph turns into an interior $C^0$ bound on the \textit{gradient} of the solution as we start flowing, so parabolic regularity theory springs into action immediately.
A similar theory can be found in Colding-Minicozzi \cite{CM} and generalisations were given by Andrews-Clutterbuck \cite{AC} and Nagase-Tonegawa \cite{nagase2009interior}. 
See also Lauer \cite{lauer}.
The case of $L^p_{loc}$ initial data for $p>1$ was addressed by Chou-Kwong \cite{CK}.
Just as for locally Lipschitz initial data, for continuous or $L^p$ ($p>1$) initial data these authors are obtaining full quantified regularity instantaneously. That is, their work implies $C^k$ estimates at time $t$ depending only on $t>0$ and a bound on the relevant local norm of  the initial data.

In this paper we do something that is qualitatively different to the papers above. We weaken the regularity class of the initial data to the extent that no \textit{a priori} regularity estimate could ever be given in terms of $t>0$ and the corresponding local norm of the initial data. Instead, our theorems will establish that there is a quantified initial period of the flow in which no estimates 
could hold, but after which parabolic regularity is switched on.
The general \textit{delayed parabolic regularity} principle we propose can be loosely phrased as:

\begin{quote}
Consider two disjoint proper embedded curves $\ga_1$ and $\ga_2$ in the plane, evolving under curve shortening flow for $t\in [0,T)$ and bounding an evolving connected region of finite area $A_0$, with $\ga_1$ and $\ga_2$ either both compact or both noncompact.
Define the threshold time $\tau$ to be $\frac{A_0}{\pi}$.
Then after flowing beyond  the threshold time we expect to start being able to control the regularity of $\ga_1$ at time $t>\tau$ in terms of $t$, $A_0$ and the regularity of $\ga_2$ at time $t$. 
\end{quote}

\begin{figure} 
\centering
\includegraphics[width=\textwidth]{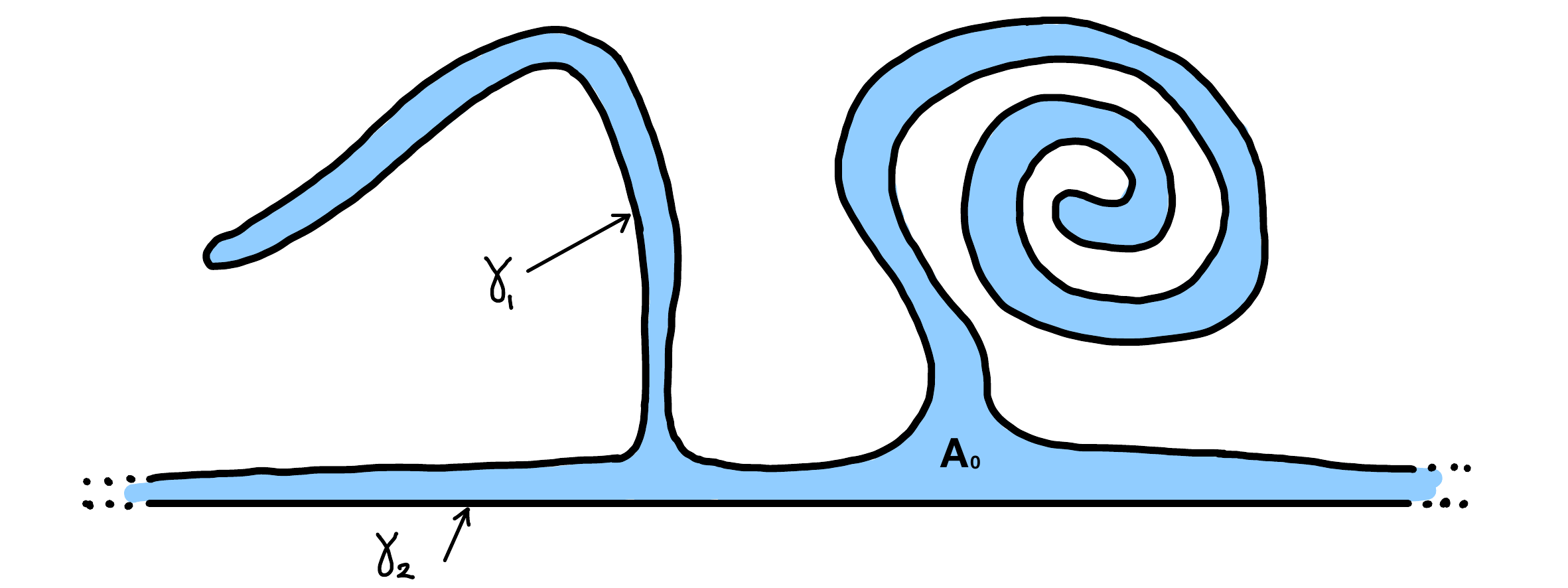}
\caption{Illustration of the delayed parabolic regularity principle}
\label{A_0_graph_general}
\end{figure}

We illustrate a possible scenario for the general delayed parabolic regularity principle in Figure \ref{A_0_graph_general}, where we have a wild curve $\ga_1$ controlled only by knowing the area $A_0$ between it and a regular curve $\ga_2$, which in this case we take to be  a (static) straight line. We expect to control the regularity of the evolution of $\ga_1$ quantifiably after time $\frac{A_0}{\pi}$, but not before.

In this work we develop and prove precise formulations of this principle in the case of graphical curve shortening flow, with $\ga_2$ being the (static) $x$-axis, as in 
Figure \ref{A_0_graph}. 
Before giving any regularity statement, we demonstrate that no traditional regularity estimate can hold prior to time $\tau$:

\definecolor{goodnotesblue}{HTML}{239CFF}

\begin{figure}
\centering
\begin{tikzpicture}

\draw[fill, domain=-5:5, samples=100, name=graph, goodnotesblue!50] 
plot (\x, {
{1/(1+(\x+1)^2) + 4/(1+(\x-2)^2)^2}
}) -- (5,0) -- (-5,0);

\draw[thick, domain=-5:5, samples=100, name=graph] 
plot (\x, {
{1/(1+(\x+1)^2) + 4/(1+(\x-2)^2)^2}
});

\draw [thick, name=xaxis] (-5,0) -- (5,0);
\draw [->] (0,0) -- (0,5) node[left]{$y_0$};

\path (2,1) node{$A_0$};

\draw [thin, ->] (-4,-1) node[below]{$\ga_2$} -- (-3.5,0);

\draw [thin, ->] (-3,2) node[above]{$\ga_1$} -- (-1.5,{1/(1+(-1.5+1)^2) + 4/(1+(-1.5-2)^2)^2});

\end{tikzpicture}
\caption{Graph to evolve under curve shortening flow enclosing area $A_0$.}
\label{A_0_graph}
\end{figure}

\begin{thm}[{No regularity until time $\tau$}]
\label{baby_uncontrollable_thm}
Suppose  $A_0>0$ and define $\tau:=\frac{A_0}{\pi}$.
Then there exists a sequence of 
Lipschitz functions $y_0^n:\R\to[0,\infty)$ with $\|y_0^n\|_{L^1}=A_0$ for each $n\in\N$, such that the unique Ecker-Huisken solutions $y^n:\R\times [0,\infty)\to\R$ of \eqref{GCSF} starting with $y_0^n$ satisfy that for every $t_0\in (0,\tau]$ we have 
$$\|y^n(\cdot,t_0)\|_{L^\infty}\to\infty$$
and 
$$\|y^n_x(\cdot,t_0)\|_{L^\infty}\to\infty$$
as $n\to\infty$.
\end{thm}

A more refined picture will be given in Theorem \ref{refined_ex} once we have recalled some basic facts about curve shortening flow on sectors.

In contrast, for times beyond $\tau$, we will obtain estimates on $y_x$, and hence on $y$ and all of its derivatives. 
An immediate question is how the threshold time $\tau$ will manifest itself in a proof or estimate. We develop two methods of proof in this paper, with a threshold time appearing in different ways in each. One way is indicated by the following theorem, in which a gradually improving upper bound for $\arctan |y_x|$ suddenly becomes non-vacuous at the time it drops below $\pi/2$; we always consider $\arctan$ as a function to $(-\frac{\pi}2,\frac{\pi}2)$.

\begin{thm}[{Delayed $L^1$-Lipschitz smoothing estimate}]
\label{est_thm_1}
Suppose $y_0:\R\to[0,\infty)$ is a locally Lipschitz function with 
$$A_0:=\|y_0\|_{L^1}\in (0,\infty).$$
Then  the unique Ecker-Huisken solution $y:\R\times [0,\infty)\to\R$ of \eqref{GCSF} starting with $y_0$ satisfies
\beql{first_arctan_est}
\arctan|y_x(x,t)|\leq \frac{A_0}{4t}+\frac{\pi}{4},
\eeq
for all $x\in \R$ and $t>0$.
In particular, for all $t>\tau:=\frac{A_0}{\pi}$, we obtain the a priori bound
$$|y_x(x,t)|\leq \tan\left(\frac{A_0}{4t}+\frac{\pi}{4}\right)$$
for all $x\in \R$. 
\end{thm}


It is important to digest that our gradient bound is not depending on any initial gradient bound and not depending on any initial $C^0$ bound.

The estimate in Theorem \ref{est_thm_1} is sharp in the sense that $\tau$ is sharp and the bound for $|y_x|$ just beyond that time is effectively best possible. 
The estimate is similar to what we obtain by a Harnack argument in Section \ref{basic_harnack_sect}.
However, we will derive \eqref{first_arctan_est} from a slightly more complicated estimate in Theorem \ref{refine_thm2} that 
is stronger for larger times and 
whose proof will invoke Angenent's crossing principle \cite{angenent1991parabolic} in a novel way (see Section \ref{crossing_arg_sect} for some alternative applications of \cite{angenent1991parabolic} in earlier papers).
Indeed, we will give a bound for the solution $y$ itself beyond the threshold time, and combine it with an estimate that controls $|y_x|$ in terms of $y$ and $t$, for arbitrary $t>0$, which is a sharp form of the estimates of Evans-Spruck \cite[\S 5.2]{ES3}.

The statement of these refinements, as well as all further discussion of 
Theorem \ref{baby_uncontrollable_thm}, will revolve around one distinguished solution of curve shortening flow, namely the evolution of 
a straight line that has been bent by 90 degrees at one point.
In Figure \ref{RAW_fig} we illustrate the flow starting with the union of the positive $y$-axis and the non-negative $x$-axis. As we recall in Sections \ref{polar_graphical_sect} and \ref{RAW_sect}, this flow will be self-similar and can be given as a solution of graphical curve shortening flow for positive times that can be written
\beql{W_self_sim}
y(x,t)=\w(x,t):=t^{\half}W(xt^{-\half}),
\eeq
where $W:(0,\infty)\to(0,\infty)$ represents the solution at time $t=1$ (see Figure \ref{RAW_fig}).

\begin{figure}
\centering
\begin{tikzpicture}

\draw [->] (0,0) -- (5,0) node[below]{$x$} ;
\draw [->] (0,0) -- (0,5);

\draw[very thick, domain=1:5.5, samples=50] 
plot (\x, {
(1/(\x))
});
\draw[very thick, domain=1:5.5, samples=50] 
plot ( {(1/(\x))},\x) node[left]{$W(x)$};

\draw[domain=0.5:5, samples=50] 
plot (\x, {
(1/(2*\x))
});
\draw[domain=0.5:5, samples=50] 
plot ( {(1/(2*\x))},\x);

\draw[domain=1:5, samples=50] 
plot (\x, {
(3/(2*\x))
});
\draw[domain=1:5, samples=50] 
plot ( {(3/(2*\x))},\x) ;

\end{tikzpicture}

\caption{Qualitative picture of curve shortening flow evolving a bent line.}
\label{RAW_fig}
\end{figure}

We will give precise properties of $W$ in Lemma \ref{RAW_lem}. For now we can proceed using the illustration in Figure \ref{RAW_fig} and the facts that $W=W^{-1}$ and the area between the graph of $W$ and the $x$ and $y$ axes turns out to be $\frac{\pi}2$. 
Let $\si:(0,\infty)\to (0,\frac{\pi}2)$  be the area of the tail as in Figure \ref{sig_A0_A1_fig}, i.e.,
\beql{si_def}
\si(x):=\int_x^\infty W(s)ds,
\eeq
and let $\ca_0:(0,\infty)\to (0,\frac{\pi}2)$ and $\ca_1:(0,\infty)\to (0,\frac{\pi}2)$ be the areas in Figure \ref{sig_A0_A1_fig}, i.e.,
\beql{ca0_def}
\ca_0(x):=\half xW(x)+\si(x)
\eeq
and
\beql{a0a1_juggling}
\ca_1(x):=\frac{\pi}2-\ca_0(x)=\ca_0(W(x)).
\eeq
By \eqref{si_def} and \eqref{ca0_def} we have
\beql{various_ids}
\si'(x)=-W(x)\qquad \text{ and }\ca_0'(x)=\half(xW'(x)-W(x)).
\eeq

\begin{figure}
\centering
\begin{tikzpicture}

\draw[fill, domain=1.5:5, samples=50, name=graph, gray!30] 
(1.5,0) -- plot (\x, {(1/(\x))}) -- (5,0);

\draw [->] (0,0) -- (5,0); 
\draw [->] (0,0) -- (0,5) node[right]{$\ W(x)$};

\draw[very thick, domain=1:5, samples=50] 
plot (\x, {
(1/(\x))
});
\draw[very thick, domain=1:5, samples=50] 
plot ( {(1/(\x))},\x) ;

\path (2,0.26) node{$\scriptstyle\si(x)$};
\draw (1.5,2/3) -- (1.5,0) node[below]{$x$};

\end{tikzpicture}
\qquad
\begin{tikzpicture}

\draw[fill, domain=1.5:5, samples=50, name=graph, gray!30] 
(1.5,0) -- plot (\x, {(1/(\x))}) -- (5,0);

\draw[pattern=crosshatch, pattern color=black!20, domain=2/3:5, samples=50, ultra thin] 
(0,0) -- plot ( {(1/(\x))},\x) -- (0,5);

\fill[gray!30] (0,0) -- (1.5,2/3) -- (1.5,0);

\draw [->] (0,0) -- (5,0); 
\draw [->] (0,0) -- (0,5) node[right]{$\ W(x)$};

\draw[very thick, domain=1:5, samples=50] 
plot (\x, {
(1/(\x))
});
\draw[very thick, domain=1:5, samples=50] 
plot ( {(1/(\x))},\x) ;

\path (1.5,0.26) node{$\scriptstyle\ca_0(x)$};

\path (0.55,0.8) node{$\scriptstyle\ca_1(x)$};

\draw [dotted, thick] (1.5,2/3) -- (1.5,0) node[below]{$x$};

\draw (0,0) -- (1.5,2/3);

\end{tikzpicture}

\caption{Illustration of $\si(x)$, $\ca_0(x)$ and $\ca_1(x)$.}
\label{sig_A0_A1_fig}
\end{figure}

The refinement of Theorem \ref{est_thm_1} will incorporate the following variant of the estimates of Evans-Spruck \cite[\S 5.2]{ES3}, which controls the gradient $|y_x|$ of a solution of \eqref{GCSF} in terms of the height $y$, and is operational instantaneously. That is, we do not have to wait for a threshold time.

\begin{thm}[Height controls gradient]
\label{ES_improvement_thm}
Suppose that $y:\R\times (0,T)\to (0,\infty)$ is a positive solution  of graphical curve shortening flow \eqref{GCSF}. 
Then for all  $t\in (0,T)$ and all $x\in \R$, we have
\beql{HCG_est}
|y_x|\leq \frac{1}{\tan\ca_0(t^{-\half}y)}\equiv\tan\ca_1(t^{-\half}y)
\equiv\tan\ca_0(W(t^{-\half}y))
\eeq
at $(x,t)$.
\end{thm}
Although this theorem is sharp, it will imply the following consequence, which may be more digestible and is still sharp for large $y$.
\begin{cor}[{cf. Evans-Spruck \cite[\S 5.2]{ES3}}]
\label{ES_imp_cor}
In the setting of Theorem \ref{ES_improvement_thm}, there exists a universal $C_1<\infty$
such that 
$$|y_x|\leq C_1 \left(\frac{y^2}{t}\right)^{\half}\exp\left[\frac{y^2}{4t}\right]+C_1$$
at $(x,t)$.
\end{cor}

Returning to estimates that can only hold after a threshold time has elapsed, we will prove:

\begin{thm}[Height control after threshold time]
\label{refine_thm1}
In the setting of Theorem \ref{est_thm_1}, for all $t>\tau:=\frac{A_0}{\pi}$ we have
\beql{delayed_height_est}
y(x,t)\leq t^{\half}W\left[\si^{-1}\left(\frac{A_0}{2t}\right)\right]
\eeq
for all $x\in \R$.
\end{thm}

By approximating the increasing function $W\circ \si^{-1}:(0,\frac{\pi}2)\to (0,\infty)$, we will show that the estimate \eqref{delayed_height_est} implies the following.
\begin{cor}
\label{refine_cor1}
In the setting of Theorem \ref{est_thm_1}, for all $t>\tau:=\frac{A_0}{\pi}$ we have
$$y\leq 2 t^\half \sqrt{-\log\left(\frac{\pi}2-\frac{A_0}{2t}\right)}$$
for $t> \frac{A_0}\pi$ with $\frac{\pi}2-\frac{A_0}{2t}$ sufficiently small.
\end{cor}

For many purposes we are happy to wait until time $2\tau$,  well beyond the threshold time $\tau$, in order to obtain a simpler estimate. Because $W\circ \si^{-1}$ is increasing, Theorem \ref{refine_thm1} immediately implies:
\begin{cor}
\label{refine_cor}
In the setting of Theorem \ref{est_thm_1}, for all $t\geq 2\tau:=\frac{2A_0}{\pi}$ we have
\beql{delayed_height_est_basic}
y(x,t)\leq C t^{\half}
\eeq
for universal $C<\infty$ and all $x\in \R$.
\end{cor}

Again, we emphasise that all these height bounds depend only on area and time, and not on any other information about the initial data.

Theorems \ref{ES_improvement_thm} and \ref{refine_thm1}, and Corollaries \ref{ES_imp_cor} and \ref{refine_cor1} will be proved in Section \ref{crossing_arg_sect}.

The control of Theorem \ref{refine_thm1} can  be combined with Theorem \ref{ES_improvement_thm} to give 
our sharp delayed gradient estimate. In order to state it most cleanly, 
we define the function $F:(0,\frac{\pi}2)\to (0,\frac{\pi}2)$ by
$$F(a)=\ca_0\left[\si^{-1}(a)\right].$$ 
Thus for the tail of area $a$,  $F(a)$ is the  area of the grey region  in Figure \ref{F_a_fig}.
\begin{figure}
\centering
\begin{tikzpicture}[xscale=1.3,yscale=1.3]

\draw[fill, domain=1.5:5, samples=50, name=graph, gray!15] 
(1.5,0) -- plot (\x, {(1/(\x))}) -- (5,0);

\draw[pattern=crosshatch, pattern color=black!40, domain=1.5:5, samples=50, name=graph] 
(1.5,0) -- plot (\x, {(1/(\x))}) -- (5,0);

\fill[gray!15] (0,0) -- (1.5,2/3) -- (1.5,0);

\draw [->] (0,0) -- (5,0) node[below]{$x$} ;
\draw [->] (0,0) -- (0,5) node[left]{$W(x)$};

\draw[very thick, domain=1:5, samples=50] 
plot (\x, {
(1/(\x))
});
\draw[very thick, domain=1:5, samples=50] 
plot ( {(1/(\x))},\x) ;

\path (2,0.26) node{$a$};

\draw (0,0) -- (1.5,2/3);

\end{tikzpicture}
\caption{
If the hatched region has area $a$ then $F(a)$ is the area of the entire grey region, including the hatched area.}
\label{F_a_fig}
\end{figure}
In the setting of Theorem \ref{est_thm_1}, we have by the strong maximum principle that $y>0$ for $t>0$, and so Theorem \ref{ES_improvement_thm} tells us 
that
$$|y_x|\leq \tan\ca_0(W(t^{-\half}y)).$$
For $t>\tau$, Theorem \ref{refine_thm1} then implies that 
$$|y_x|\leq \tan\ca_0\left[\si^{-1}\left(\frac{A_0}{2t}\right)\right]
=\tan\left[F\left({\frac{A_0}{2t}}\right)\right].$$
To summarise, we have proved

\begin{thm}[Refined gradient estimate]
\label{refine_thm2}
In the setting of Theorem \ref{est_thm_1}, for all $t>\tau:=\frac{A_0}{\pi}$ we have
$$|y_x(x,t)|\leq \tan\left[F\left({\frac{A_0}{2t}}\right)\right]$$
for all $x\in \R$.
\end{thm}

Clearly $F(a)$ is monotonically increasing in $a$ and is a diffeomorphism. 
It extends to a homeomorphism from $[0,\frac{\pi}2]$ to itself by setting $F(0)=0$ and 
$F(\frac{\pi}2)=\frac{\pi}2$.
We can also compute\footnote{
Because $F(\si(x))=\ca_0(x)$, the identities of \eqref{various_ids} imply that 
$F'(\si(x))=\frac{\ca_0'(x)}{\si'(x)}=\frac{xW'(x)-W(x)}{-2W(x)}=\half(1+\frac{x(-W'(x))}{W(x)})\geq \half$. } 
for $a\in (0,\frac{\pi}2)$ that
$$F'(a)>\half,$$
which, because $F(\frac{\pi}2)=\frac{\pi}2$, immediately gives 
$$F(a)< \frac{\pi}4+\frac{a}2.$$
In particular, Theorem \ref{refine_thm2} strictly generalises Theorem \ref{est_thm_1}.

Another use of the function $W$ is to give an elaboration of 
Theorem \ref{baby_uncontrollable_thm}.

\begin{figure}
\centering
\begin{tikzpicture}

\draw [very thin, <->, name=xaxis] (-5,0) -- (5,0);
\draw [very thin, ->] (0,0) -- (0,4) node[left]{$y_0^n$};

\draw[thick] (-5,0) -- (-1/3,0) node[below]{$-\frac{1}{n}$} -- (0,3) node[left]{$n$} -- 
(1/3,0) node[below]{$\frac{1}{n}$} -- (5,0);

\end{tikzpicture}
\caption{Graph of $y_0^n$ in Theorem \ref{refined_ex}.}
\label{witches_hat}
\end{figure}

\begin{thm}[Curve shortening flow from a delta function]
\label{refined_ex}
For each $n\in\N$, define the \emph{witch's hat} Lipschitz function $y_0^n:\R\to [0,\infty)$ (see Figure \ref{witches_hat}) by
$$y_0^n(x)=\left\{
\begin{aligned}
0 \qquad & {\textstyle \text{if }|x|\geq\frac{1}{n}}\\
n(1-n|x|) \qquad  & {\textstyle \text{if }|x|<\frac{1}{n}},\\
\end{aligned}
\right.$$
%
%
noting that $\|y_0^n\|_{L^1(\R)}=1$ for all $n\in \N$
so the threshold time is $\tau:=\frac{1}{\pi}$ for each $n\in\N$.
Let $y^n:\R\times [0,\infty)\to \R$
be the  Ecker-Huisken solutions to graphical curve shortening flow \eqref{GCSF} 
evolving from $y_0^n$.
Then after passing to a subsequence, $y^n$ converges smoothly locally on 
$(\R\times (0,\infty))\setminus(\{0\}\times (0,\tau])$ to a limit solution $y^\infty$. 
Moreover, for $t\in (0,\tau]$ we have
$$y^\infty(x,t)=\w(|x|,t)\qquad  \text{ for }x\neq 0$$
%
%
where $\w(x,t):=t^{\half}W(xt^{-\half})$ as in \eqref{W_self_sim}.
In particular, the limit $y^\infty$ is singular and unbounded for $t\in (0,\tau]$, but smooth and bounded for $t>\tau$.
\end{thm}

Essentially the solution $y^\infty$ has initial data given by a delta function. As time advances from $0$ to $\tau$ the mass leaks out into the smooth solution away from $x=0$.
For $t_0\in (0,\tau)$ we can make $y^n(\cdot, t_0)$ as uncontrolled as we like by taking $n$ large enough, but beyond time $\tau$ the flows $y^n$ 
have uniformly controlled regularity as we know they must by Theorem \ref{est_thm_1}.
The behaviour of the limit $y^\infty$ is somewhat reminiscent of the Ricci flow constructed in \cite{hct}. 
%
Theorem \ref{refined_ex} will be proved in Section \ref{delay_ex_sect}.

In a different direction, our time-delayed estimates can be applied to prove instantaneous regularisation for slightly more constrained initial data that is nevertheless still allowed to be unbounded.

\begin{thm}[{Instantaneous $L^p$-Lipschitz smoothing estimate for $p>1$}]
\label{Lp_reg}
Suppose $y_0:\R\to\R$ is a locally Lipschitz function in $L^p(\R)$ for some $p>1$.
Then  the unique Ecker-Huisken solution $y:\R\times [0,\infty)\to\R$ of \eqref{GCSF} starting with $y_0$ satisfies
\beql{desired_y_bd_Lp}
|y(x,t)|\leq C t^{-\frac1{p-1}}\|y_0\|_{L^p(\R)}^{\frac{p}{p-1}}
\eeq
for all $x\in \R$, all $0<t<\|y_0\|_{L^p(\R)}^{\frac{2p}{p+1}}$,
and universal $C$.
\end{thm}
One could also obtain explicit bounds on $y$ for larger $t$. We prove Theorem \ref{Lp_reg} in Section \ref{Lpsect}.

Our time-delayed estimates can be applied to prove existence of the curve shortening flow
with rough data.
The (standard) general approach is to approximate the rough initial data by smooth data, flow the smoothed data, and try to take a limit of the flows. The crucial point is that our estimates are precisely those required to give compactness, allowing us to extract a limit of a subsequence of these flows.

For context, recall that Ilmanen \cite[Example 7.3]{ilmanen} proved existence of curve shortening flow starting essentially with the graph of the function $y_0(x)=x^{-\be}$ for $\be\in (0,1)$, making the spike contract instantly.
(See also Peachey \cite{peachey} and Bourni-Reiris \cite{BR}.)
Here we flow a general $L^1$ function.

\begin{thm}[{Flowing from $L^1(\R)$ initial data}]
\label{L1_exist_thm}
Let $y_0\in L^1(\R)$.
Then there exists a smooth solution 
$y:\R\times (0,\infty)\to \R$ to 
\eqref{GCSF} 
such that 
$$y(\cdot,t)\to y_0\qquad \text{ in } L^1(\R) \text{ as }t\downto 0.$$
\end{thm}

We will prove this theorem in Section \ref{L1_exist_sect}. There we also remark that generalisations of the techniques in this paper can handle merely $L^1_{loc}$ data, 
or even data that is a nonatomic Radon measure. 
Each of these existence results answers questions raised by Chou-Kwong \cite{CK}, who handled the case of $L^p_{loc}$ data with $p>1$, when instantaneous smoothing is available.
Details will be found in \cite{arjun_thesis}.

\brmk
\label{with_yin_rmk}
Although our results in this paper have generally a different flavour to previous curve shortening flow results, there is an intriguing resemblance to the  Ricci flow theory that originates from the estimates of the second author and H. Yin in \cite{TY1}.
This, in turn, is related to estimates arising in K\"ahler Ricci flow. In particular, one can compare with the work of Di Nezza and Lu \cite{DN_L}. In \cite{ST3} we develop further the Harnack estimates introduced in this paper and make contact with work of Neves on Lagrangian mean curvature flow \cite{neves}.
In order to prove generalisations of the results in this paper, we expect the monotonicity property discovered in \cite{ST_modulus} to be useful.
\ermk

\section{The Harnack estimate for graphical curve shortening flow}
\label{basic_harnack_sect}

In this section we introduce  a Harnack quantity for graphical curve shortening flow that is already enough to illustrate delayed parabolic regularity and to prove existence of solutions with very general rough initial data. It is sharp in certain circumstances, but in a different way to Theorem \ref{est_thm_1}.
We emphasise that our Harnack estimate does not require any convexity assumption.

\begin{thm}
\label{est_thm_harnack}
Suppose $y_0:\R\to[0,\infty)$ is a  Lipschitz function with compact support, and define 
$$A_0:=\|y_0\|_{L^1}.$$
Consider the unique Ecker-Huisken solution $y:\R\times [0,\infty)\to\R$ of \eqref{GCSF} starting with $y_0$. Then $y\geq 0$, the total area remains constant, i.e., $\|y(\cdot,t)\|_{L^1(\R)}=A_0$ for all $t\geq 0$, 
and the \emph{accumulated area}
\beql{acc_area_fn}
\ca(x,t):=\int_{-\infty}^x y(s,t)ds
\eeq
(see Figure \ref{acc_area_fig}) defines a function  $\ca:\R\times [0,\infty)\to [0,A_0]$, which satisfies
\beql{second_arctan_est}
\arctan y_x(x,t)\leq \frac{\ca(x,t)}{2t}+\frac{\pi}4.
\eeq
for all $x\in \R$ and $t>0$.
\end{thm}

\begin{figure}
\centering
\begin{tikzpicture}

\draw[fill, domain=-5:-0.9, samples=100, name=graph, blue!8] 
plot (\x, {
{1.5/(1+(\x+0.5)^2) + 3/(1+(\x-2)^2)^2}
}) -- (-0.9,0);

\draw[thick, domain=-5:5, samples=100, name=graph] 
plot (\x, {
{1.5/(1+(\x+0.5)^2) + 3/(1+(\x-2)^2)^2}
});

\draw [thick, name=xaxis] (-5,0) -- (5,0);
\draw [->] (0,0) -- (0,4) node[left]{$y(\cdot,t)$};

\path (-1.4,0.35) node{$\scriptstyle\ca(x,t)$};

\draw (-0.9,{1.5/(1+(-0.9+0.5)^2) + 3/(1+(-0.9-2)^2)^2}) -- (-0.9,0) node[below]{$x$};

\end{tikzpicture}
\caption{The accumulated area $\ca(x,t)$.}
\label{acc_area_fig}
\end{figure}

We will note shortly that for each $t> 0$, the gradient $y_x(x,t)$ converges to zero as $x\to-\infty$. Theorem \ref{est_thm_harnack} tells us how fast control on the gradient propagates into the interior of the domain in terms of the area under the graph.
It tells us that there is a front at $x=x_0(t)$, where $\ca(x_0(t),t)=\frac{\pi t}{2}$, to the left of which we have gradient control and to the right of which we have nothing.

\begin{proof}
A straightforward barrier argument using slowly shrinking large circles below the $x$-axis shows that the solution $y$ must remain non-negative.
Using slowly shrinking large circles \textit{above} the $x$-axis shows that the solution $y(x,t)$ continues to decay to zero as $|x|\to\infty$ also for $t>0$. 
If we choose the circles sufficiently carefully we find that for all $T>0$
there exist $C<\infty$ and $x_0<\infty$ depending only on the initial data $y_0$ and $T$ such that 
\beql{required_y_bd}
y(x,t)\leq \frac{C}{x^2},
\eeq
for all $t\in [0,T]$ and all $x$ with $|x|\geq x_0$.
To see this for $x\geq x_0$, first write $\bar{y}:=\max y_0$ and $\bar{x}:=\sup\{x\,:\,y_0(x)\neq 0\}$.
We will argue that circles 
tangent to (and above) the $x$-axis at $(x,0)$ of radius $r=\frac{x^2}{3\bar{y}}$ will be disjoint from the  graph of $y_0$ for large enough $x$ 
except at the point of tangency. Take any point $(a,b)$ on the graph of $y_0$, but not on the $x$-axis. 
By constraining $x_0$ to be large enough, depending on $\bar{y}$ and $\bar{x}$, we may assume that $x>\bar{x}\geq a$ and $b\leq \bar{y}< \frac{x^2}{3\bar{y}}$.
Then the squared distance from $(a,b)$ to the circle's centre $(x,r)$, with $r=\frac{x^2}{3\bar{y}}$,
is 
\beqa
(x-a)^2+(r-b)^2 &\geq 
(x-\bar{x})^2+(r-\bar{y})^2\\
& \geq r^2 + x^2 -2\bar{y}r +\text{lower order terms in }x\\
& \textstyle \geq r^2 + x^2 -\frac{2}{3}x^2 +\text{lower order terms in }x\\
&> r^2
\eeqa
provided we constrain $x_0$ to be sufficiently large (independently of $a$ and $b$) so that the remaining quadratic term $x^2/3$ dominates the lower order terms. Thus $(a,b)$ lies outside the circle.

These circles, shrinking under curve shortening flow, will provide suitable barriers to obtain \eqref{required_y_bd}.
The details above are not essential because a dramatically stronger estimate will arise as a by-product of Section \ref{RAW_sect}, in Lemma \ref{rapid_decay_cpt_spt_lem}.
One can also obtain sufficient control by using so-called Angenent ovals or grim reapers as barriers.

A consequence of the decay \eqref{required_y_bd} is that the area between the graph of $y(\cdot,t)$ and the $x$-axis remains finite.
Additionally, since we know by Ecker-Huisken \cite[Corollary 3.2]{EH1} that $y(\cdot,t)$ is Lipschitz for all $t>0$ with the same Lipschitz constant as for $y_0$, we may invoke parabolic regularity theory on space-time regions $(x_1-1,x_1+1)\times (0,T]$, where $|x_1|$ is sufficiently large so that $y_0$ is identically zero throughout $(x_1-1,x_1+1)$, 
and sufficiently large so that $|x_1|>|x_0|+1$.
We deduce that all derivatives of $y$ decay quadratically analogously to \eqref{required_y_bd}. In particular, there exists $C<\infty$ such that 
$$|y_x(x,t)|+|y_t(x,t)|\leq \frac{C}{x^2},$$
provided that $t\in [0,T]$ and $|x|$ is sufficiently large, independently of $t$. 

This regular behaviour at spatial infinity justifies the computations we make below. For example, the total area under $y$ remains constant:
$$\frac{d}{dt}\int_{-\infty}^\infty y(s,t)ds=\int_{-\infty}^\infty y_t(s,t)ds
=\int_{-\infty}^\infty (\arctan(y_x))_x(s)ds
=0.$$
By the decay \eqref{required_y_bd} and the local uniform convergence of $y(\cdot,t)$ to $y_0$ as $t\downto 0$, the area is continuous at $t=0$
and so the total area remains as $A_0$.

These considerations show that the accumulated area function $\ca:\R\times [0,\infty)\to [0,A_0]$ defined in \eqref{acc_area_fn} makes sense.
Note that for fixed $t\geq 0$, the function $\ca(\cdot,t)$ increases from $0$ to $A_0$ as we pass from $-\infty$ to $\infty$.
We have $\ca_x=y$, and so $\ca_{xx}=y_x=:\al$.
We then compute at $(x,t)$
\beqa
\label{At_eq}
\ca_t & =\int_{-\infty}^x y_t(s)ds\\
&= \int_{-\infty}^x (\arctan(y_x))_x(s)ds\\
&= \arctan \al\\
&= \arctan \ca_{xx}.
\eeqa
Define the Harnack quantity $\h:\R\times (0,\infty)\to \R$ by 
\beqa
\h&:=\ca-2t\arctan\ca_{xx}\\
&=\ca-2t\arctan\al.
\eeqa
By the boundedness of $\arctan$ and continuity of area at $t=0$ (similarly to above)
$\h$ extends continuously to a function $\h:\R\times [0,\infty)\to \R$, and 
$\h(\cdot,0)$ takes values in $[0,A_0]$.
Moreover, we compute for $t>0$ that
\beqa
\h_t&=\ca_t-2\arctan\al-\frac{2t}{1+\al^2}\al_t\\
&=-\arctan\al-\frac{2t\al_t}{1+\al^2}
\eeqa
by \eqref{At_eq}, and
\beqa
\h_{xx}&=\ca_{xx}-2t(\arctan\al)_{xx}\\
&=\al-2t\al_t
\eeqa
by \eqref{alpha_eq}, so
\beql{H_eq}
\h_t-\frac{\h_{xx}}{1+\al^2}=-\arctan\al-\frac{\al}{1+\al^2}\in
\left(-\frac{\pi}{2},\frac{\pi}{2}\right).
\eeq

Note here that the odd function $f(\al):=\arctan\al+\frac{\al}{1+\al^2}$
satisfies
\beqa
f'(\al)&=\frac{1}{1+\al^2}+\frac{1}{1+\al^2}-\frac{2\al^2}{(1+\al^2)^2}\\
&=\frac{2}{(1+\al^2)^2}>0,
\eeqa
so $f$ is maximised as $\al\to\infty$.

Observe that $\h(x,t)$ converges to $0$ as $x\to -\infty$ and converges to $A_0$ as $x\to\infty$, locally uniformly in $t\in [0,\infty)$.
The maximum principle then applies to \eqref{H_eq} to give that at time $t\in [0,\infty)$ we have
\beq
\textstyle \h\in [-\frac{t\pi}{2},A_0+\frac{t\pi}{2}].
\eeq
Unravelling the definition of $\h$, this gives
upper and lower gradient bounds
\beq
\arctan\al(x,t)\leq \frac{\ca(x,t)}{2t}+\frac{\pi}4
\eeq
and 
\beq
\arctan\al(x,t)\geq -\left[\frac{A_0-\ca(x,t)}{2t}+\frac{\pi}4\right].
\eeq
The lower gradient bound could alternatively be derived by applying the upper gradient bound to the solution $(x,t)\mapsto y(-x,t)$.
\end{proof}

\begin{rmk}
\label{even_harnack_rmk}
One consequence of the gradient estimate of Theorem \ref{est_thm_harnack} is that if we have a solution that is even, i.e., $y(x,t)=y(-x,t)$, with $y_x\geq 0$ for negative $x$, 
then because $\ca(x,t)\leq A_0/2$ for $x\leq 0$, we have for $t\geq \tau:=A_0/\pi$ that
$$\arctan\al(x,t)\leq \frac{A_0}{4t}+\frac{\pi}4$$
i.e., we suddenly obtain gradient bounds as we pass beyond
the threshold time $\tau$.
We will use this fact in Section \ref{delay_ex_sect}.
\end{rmk}

Theorem \ref{est_thm_harnack} assumes that the initial data $y_0$ has compact support. 
However, the following lemma invokes earlier work  to remove this restriction.

\begin{lem}
\label{pres_area_lem}
Suppose $y_0:\R\to \R$ is a  locally Lipschitz function.
Let $\vph\in C^\infty(\R,[0,1])$ be a decreasing function with $\vph(x)=1$ for $x\leq 0$ and $\vph(x)=0$ for $x\geq 1$, and use it to define Lipschitz functions
$y_0^i:\R\to\R$ with compact support by 
$$y_0^i(x)=\vph(|x|-i)y_0(x),$$
so that $y_0^i \equiv y_0$ on the ball $B_i(0)$.
Then after passing to a subsequence, 
the corresponding Ecker-Huisken solutions $y^i:\R\times [0,\infty)\to\R$ of \eqref{GCSF} starting with $y_0^i$ converge smoothly locally on 
$\R\times (0,\infty)$, and locally uniformly on $\R\times [0,\infty)$, 
to the Ecker-Huisken solution $y:\R\times [0,\infty)\to\R$ of \eqref{GCSF} starting with $y_0$. 
Moreover, if $y_0\geq 0$ (so $y\geq 0$) and   $y_0\in L^1(\R)$, then 
$\|y(\cdot,t)\|_{L^1(\R)}=\|y_0\|_{L^1(\R)}$ for all $t>0$.
\end{lem}

The lemma allows us to extend results we know for compactly supported initial data to general initial data $y_0$, by working with the approximations $y_0^i$ and passing to the limit $i\to\infty$. Two example corollaries are:

\begin{cor}
\label{est_thm_harnack_cor}
Theorem \ref{est_thm_harnack} applies even if we replace the hypotheses that $y_0$ is Lipschitz and of compact support with the hypotheses that $y_0$ is locally Lipschitz
and $y_0\in L^1(\R)$.
\end{cor}

\begin{cor}[Comparison principle for locally Lipschitz initial data]
\label{comp_princ_loc_lip}
If $y^-_0$ and $y^+_0$ are locally Lipschitz functions on $\R$ such that 
$y^-_0\leq y^+_0$, then the  Ecker-Huisken solutions $y^-$ and $y^+$ 
evolving from $y^-_0$ and $y^+_0$ respectively
satisfy
$y^-\leq y^+$.
\end{cor}

\begin{proof}[{Proof of Lemma \ref{pres_area_lem}}]
By the estimates of Ecker-Huisken \cite{EH2} and parabolic regularity theory,
after passing to a subsequence, the flows $y^i$ converge smoothly locally on 
$\R\times (0,\infty)$, and locally uniformly on $\R\times [0,\infty)$ to some new solution on $\R\times [0,\infty)$ starting with $y_0$, which must be $y$ by uniqueness.

In the case that $y_0\geq 0$ and $y_0\in L^1(\R)$, 
the functions $y_0^i$ will be increasing in $i$, 
and the monotone convergence theorem tells us that 
$\|y_0^i\|_{L^1(\R)}\upto \|y_0\|_{L^1(\R)}$ as $i\to\infty$.
By the first part of Theorem \ref{est_thm_harnack}, 
for all $t>0$ the  solutions $y^i$
satisfy 
$$\|y^i(\cdot,t)\|_{L^1(\R)}=\|y_0^i\|_{L^1(\R)}\upto \|y_0\|_{L^1(\R)}$$
as $i\to\infty$.
By the comparison theorem, the  solutions $y^i$ also increase in $i$.
By the monotone convergence theorem their limit $y:\R\times [0,\infty)\to[0,\infty)$
satisfies
$\|y(\cdot,t)\|_{L^1(\R)}= \|y_0\|_{L^1(\R)}$ for each $t>0$
as required.
\end{proof}

\section{Polar graphical flows}
\label{polar_graphical_sect}

Consider the curve shortening flow starting with a line in $\R^2$ that has been bent at one point to have an interior angle 
$\be\in (0,\pi]$. As this is the graph of a Lipschitz function, the original theory of Ecker-Huisken \cite{EH1} tells us that there exists a graphical solution for all time. Chou-Zhu \cite{CZ} tell us that this is the unique graphical solution.
If \( \beta = \pi \) then this is the trivial static line solution
and we temporarily exclude this case in the discussion.

\begin{figure}
\centering
\begin{tikzpicture}[scale=1.2]

\draw [->] (0,0) -- (2.8*1,2.8*1.732);
\draw [->] (0,0) -- (2.8*2,0);

\draw[very thick, samples=40,domain=0.6:59.4] plot ({cos(\x)/(sin(3*\x))^(0.5)}, {sin(\x)/(sin(3*\x))^(0.5)} );

\draw[-] (0.2,0.2) node[right]{$\be$};
\draw[samples=40,domain=0:60] plot ({cos(\x)/1.5}, {sin(\x)/1.5}); 

\end{tikzpicture}
\caption{Qualitative picture of the $\be$-wedge solution.}
\label{figbwed}
\end{figure}

By uniqueness, the resulting solution must necessarily be self-similar and has been studied by many authors, e.g.~Ishimura \cite{ishimura}. We will call it the $\be$-wedge solution. 
%
Ishimura \cite{ishimura}  tells us that the 
$\be$-wedge solution can be written as a so-called polar graph.
More precisely, instead of considering curve shortening flows of graphs $y(x)$, we can 
write $x+iy=re^{i\th}$ and consider flows of polar graphs $r(\th)$. 
For $t>0$, the $\be$-wedge solution will take the form
$$r_\be(\th,t)=t^\half R_\be(\th)$$
where $R_\be:(0,\beta)\to (0,\infty)$
is a convex function that 
satisfies 
$R_\be(\th)\to\infty$ both as $\th\downto 0$ and as $\th\upto \be$, cf.~Figures \ref{RAW_fig} and \ref{figbwed}.

Denote by $\cv_\be(\th_0,\th_1)$, for $0<\th_0<\th_1<\be$, the area of 
$$\{re^{i\th}\ :\ \th\in (\th_0,\th_1),\ 0<r<R_\be(\th)\},$$
that is, 
$$\cv_\be(\th_0,\th_1)=\int_{\th_0}^{\th_1}\frac{R_\be(\th)^2}{2}d\th,$$
and denote by
\beql{zz}
\cv_\be(\th_0,\th_1,t)=t\cv_\be(\th_0,\th_1)
\eeq
the corresponding area for $r_\be(\cdot,t)$ rather than $R_\be$.
We will derive two formulae for $\frac{d}{dt}\cv_\be(\th_0,\th_1,t)$ and equate them. Directly from 
\eqref{zz} we have 
\beql{formula_1}
\frac{d}{dt}\cv_\be(\th_0,\th_1,t)=\cv_\be(\th_0,\th_1).
\eeq
If we write the polar graph of $R_\be$ instead as the image of a unit-speed parametrisation 
$\ga_\be:\R\to\R^2$, with $\ga_\be(s)$ limiting to the lines $\th=0$ and $\th=\be$  as $s\to-\infty$ and $s\to\infty$, respectively, 
and use the fact that the curve shortening flow moves the curve in the direction of the geodesic curvature $-\ka_\be \nu$, for $\nu$ the unit normal that points more away from the origin than towards it, then we have
\beql{vol_ka}
\left.\frac{d}{dt}\cv_\be(\th_0,\th_1,t)\right|_{t=1}=-\int_{s(\th_0)}^{s(\th_1)} \ka_\be(s)ds.
\eeq
Our sign convention for $\ka_\be$ means that 
if we take any choice $\psi_\be:\R\to\R$  of the angle 
the tangent $\partial_s \gamma_\be$ makes to a fixed reference direction\footnote{
For example, $\partial_s \gamma_\be(s)=(\cos\psi_\be(s),\sin\psi_\be(s))$ when measuring \( \psi_\be \) from the \( x \)-axis.}, then 
$\ka_\be(s)=\partial_s \psi_\be(s)$. In particular, this can be integrated in \eqref{vol_ka} to give 
$$\left.\frac{d}{dt}\cv_\be(\th_0,\th_1,t)\right|_{t=1}=-(\psi_\be(s(\th_1))-\psi_\be(s(\th_0))).$$
Abbreviating 
\beql{PsiBetaDef}
\Psi_\be(\th_0, \th_1):=\psi_\be(s(\th_0)) - \psi_\be(s(\th_1)), 
\eeq
which geometrically is the angle through which  the tangent turns  when moving 
from \( \th_1 \) to \( \th_0 \), 
and combining with \eqref{formula_1}, we find that
$$\cv_\be(\th_0,\th_1)=\Psi_\be(\th_0, \th_1).$$%

We have proved the following elementary identity.
%
%
%
\begin{lem}
\label{zero_Harnack_lem}
With the notation above, the \( \be \)-wedge solution of curve shortening flow satisfies 
$$\h_\be(\th_0,\th_1):=\cv_\be(\th_0,\th_1)-\Psi_\be(\th_0, \th_1)\equiv 0.$$
\end{lem}

Taking the limit $\th_0\downto 0$ and as \( \th_1 \upto\be\)  
gives (with slight abuse of notation)
\beq
\label{vol_ss}
\cv_\be(0,\be) = \pi-\be.
\eeq


Lemma \ref{zero_Harnack_lem} saturates a Harnack inequality that 
was given in the original preprint of this paper, but which is now  discussed in \cite{ST3} in greater generality.

\section{The right-angled wedge}
\label{RAW_sect}

In this section we continue  the discussion of the self-similar solutions from the previous section in the case that $\beta=\frac{\pi}2$. Instead of viewing the solution as a polar graph, we can then view it as a standard graph. The  polar graph of $R_{\frac\pi{2}}$ becomes the graph of $W$, as discussed in Section \ref{intro}.
Recall that the graph of $W$ was illustrated in Figure \ref{RAW_fig}.

\begin{lem}[Right-angled wedge]
\label{RAW_lem}
There exists a monotonically decreasing convex diffeomorphism $W:(0,\infty)\to(0,\infty)$ with the properties that 
\begin{enumerate}
\item
\label{P1}
$y(x,t)=\w(x,t):=t^{\half}W(xt^{-\half})$ solves the graphical curve shortening flow equation \eqref{GCSF},
\item
\label{P2}
$W^{-1}=W$,
\item 
\label{P3}
$\|W\|_{L^1}=\frac{\pi}{2}$,
\item
\label{W_prime_area_item}
$-\arctan W'(x)=\ca_0(x)$, where $\ca_0:(0,\infty)\to (0,\frac{\pi}2)$ is defined 
in \eqref{ca0_def}.
\item
\label{W_decay_item}
The function $W$ decays rapidly in the sense that for $0<x<s$, we have 
\beql{W_decay}
W(s)\leq W(x)e^{(x^2-s^2)/4}.
\eeq
\item
\label{tail_item}
The tail area \eqref{si_def} decays rapidly in the sense that 
$$\si(x) 
\leq \frac{2W(x)}{x},$$
for $x>0$.
\item
\label{W_Wp_item}
As $x\to\infty$ we have
$$-W'(x)= (1+o(1))\frac{xW(x)}2.$$
\item
\label{W_lower_item_new}
As $x\to\infty$ we have
$$W(x)= \big(1+o(1)\big)2d e^{d^2/4}\frac{e^{-x^2/4}}{x^2},$$
where the distance $d$ from the origin to the graph of $W$ can be estimated by 
$d^2\leq \frac{\pi}{2}$.
\end{enumerate}
\end{lem}

\begin{proof}
Parts \ref{P1}, \ref{P2}, \ref{P3} and \ref{W_prime_area_item} follow from the discussion in Section \ref{polar_graphical_sect}.
In particular, Part \ref{P3} follows from \eqref{vol_ss} and 
Part \ref{W_prime_area_item} follows from Lemma \ref{zero_Harnack_lem} because
$0=\h_{\frac{\pi}{2}}(0,\th)=\ca_0(x)+\arctan W'(x)$, where 
$x=R_{\frac\pi{2}}(\th)\cos\th$.
For Part \ref{W_decay_item}, we use Part \ref{W_prime_area_item} to compute
\beql{for_upper}
-W'(x) \geq \arctan(-W'(x)) = \ca_0(x) = \half xW(x)+\si(x)\geq \half xW(x),
\eeq
or equivalently 
$$(-\log W)'(x)\geq \frac{x}2,$$
which can be integrated to give \eqref{W_decay}.

We use Part \ref{W_decay_item} to prove Part \ref{tail_item}.
We find that
$$\si(x)=\int_x^\infty W(s)ds\leq W(x) \int_x^\infty e^{(x^2-s^2)/4}ds.$$
By changing variables from $s$ to $r$, where $s=x+r$, 
we compute that 
$$\int_x^\infty e^{(x^2-s^2)/4}ds=\int_0^\infty e^{(-rx/2-r^2/4)} dr \leq  \int_0^\infty e^{-rx/2} dr  = \frac{2}{x}, $$
and so 
$$\si(x)\leq \frac{2W(x)}{x}$$
as required.

For Part \ref{W_Wp_item}, we note that by the definition of $\ca_0(x)$ from \eqref{ca0_def}, and by Part \ref{tail_item} we have
$$\frac{xW(x)}2\leq \ca_0(x)\leq W(x)\left[\frac{x}{2}+\frac{2}{x}\right].$$
Therefore Part \ref{W_prime_area_item} tells us not only that $W'(x)\to 0$ as $x\to\infty$ but also that
$$-W'(x)=(1+o(1))(-\arctan W'(x))=(1+o(1))\ca_0(x)=(1+o(1))\frac{xW(x)}2$$
as required.

For Part \ref{W_lower_item_new}, we start by differentiating 
Part \ref{W_prime_area_item}, using \eqref{various_ids}, to give 
\beql{part8_preamble}
\frac{2W''(x)}{1+(W'(x))^2}=-xW'(x)+W(x).
\eeq
Then we can verify simply by differentiating that
\beql{logW}
\log(1+(W')^2)=\frac{x^2+W^2}2 +\log(-xW'+W)^2+C
\eeq
for some constant $C$.\footnote{The relevant computations are:
$$\frac{d}{dx}\log(-xW'+W)^2= \frac{-2xW''}{-xW'+W}\stackrel{\eqref{part8_preamble}}{=}-x(1+(W')^2)$$
$$\frac{d}{dx}\log(1+(W')^2)=\frac{2W'W''}{1+(W')^2}\stackrel{\eqref{part8_preamble}}{=}-x(W')^2+W W'$$
$$\frac{d}{dx}\left(\frac{x^2+W^2}2\right)=x+W W'$$
}
Note that a first integral for self-similar solutions can also be found in 
\cite{halldorsson}.

\begin{figure}
\centering
\begin{tikzpicture}[scale=1.2]

\draw [->] (0,0) -- (5,0); 
\draw [->] (0,0) -- (0,5) node[left]{$W(x)$};

\draw[very thick, domain=1:5, samples=50] 
plot (\x, {
(1/(\x))
});
\draw[very thick, domain=1:5, samples=50] 
plot ( {(1/(\x))},\x);

\draw (0,0) -- (0.6,0.6) node[above left]{$d$} -- (1,1);
\draw [dotted] (1,1) -- (1,0) node[below]{$x$};

\end{tikzpicture}

\caption{Definition of $d$.}
\label{x_d_fig}
\end{figure}

If we let $d$ be the shortest distance from origin to curve, and consider the 
value $x$ so $x^2=\half d^2$, so $W(x)=x$ (see Figure \ref{x_d_fig}) then we find 
from \eqref{logW} that
$$\log 2 = \frac{d^2}2 +\log (2d^2) +C,$$
so $e^{-C}=d^2e^{d^2/2}$.
Exponentiating \eqref{logW} then gives
\beq
d^2e^{d^2/2}(1+(W')^2)=(-xW'+W)^2e^{\frac{x^2+W^2}2}.
\eeq
Using that $W'(x)=o(1)$ and  $W(x)=o(1)$ as $x\to\infty$ then gives
$$(-xW'+W)=d e^{d^2/4} e^{-\frac{x^2}4}  (1+o(1)).$$
By Part \ref{W_Wp_item}, $\frac{W}{-xW'}=o(1)$ as $x\to\infty$, so
\beqa 
(-x^2W)' & =x(-xW'-2W)\\
&=x(-xW'+W)(1+o(1))\\
&=d e^{d^2/4} xe^{-\frac{x^2}4}  (1+o(1))\\
&=-2d e^{d^2/4} \left[e^{-\frac{x^2}4}\right]'  (1+o(1)).
\eeqa
Because $x^2W(x)=o(1)$ as $x\to\infty$, we can integrate from $x$ to infinity to give
$$x^2W(x)=2d e^{d^2/4} e^{-\frac{x^2}4} (1+o(1))$$
as $x\to\infty$, as required.
To estimate $d$ we observe that the area of the biggest triangle one can squeeze under the graph of $W$ is $d^2$, as in Figure \ref{d2_fig}. Therefore $d^2\leq \frac{\pi}{2}$ by Part \ref{P3}.
\end{proof}

\begin{figure}
\centering
\begin{tikzpicture}[scale=1.2]

\draw [fill, blue!8] (0,2) -- (2,0) -- (0,0);

\draw [->] (0,0) -- (5,0); 
\draw [->] (0,0) -- (0,5) node[left]{$W$};

\draw[very thick, domain=1:5, samples=50] 
plot (\x, {
(1/(\x))
});
\draw[very thick, domain=1:5, samples=50] 
plot ( {(1/(\x))},\x);

\draw (0,2) -- (2,0) node[below]{$2x$};
\draw [dotted] (1,1) -- (1,0) node[below]{$x$};

\path (0.6,0.6) node{$d^2$};

\end{tikzpicture}

\caption{Area of the shaded triangle is $\half(2x)^2=2x^2=d^2$.}
\label{d2_fig}
\end{figure}

One simple consequence of our estimates on $W$ 
is that an Ecker-Huisken solution starting with a Lipschitz function of compact support must decay rapidly to zero as $x\to\infty$ and $x\to-\infty$, and in particular must remain in $L^1$. We can translate the solution to the left so that the initial support is in $\{x<0\}$, and apply the following.

\begin{lem}
\label{rapid_decay_cpt_spt_lem}
Suppose $y: [0,\infty)\times(0,T)\to\R$ satisfies \eqref{GCSF} and suppose that 
$y(\cdot,t)\to 0$ locally uniformly 
as $t\downto 0$.
Then 
\beql{yWexp}
|y(x,t)|\leq \w(x,t)\leq C\frac{t^{\frac{3}{2}}}{x^2} e^{-\frac{x^2}{4t}}
\eeq
for some universal constant $C<\infty$, all $t\in (0,T)$, and all $x\geq t^\half$.
\end{lem}

\begin{proof}
We claim that the hypothesis of local uniform convergence $y(\cdot,t)\to 0$ 
as $t\downto 0$ will imply uniform convergence 
and decay of $y(x,t)$ to zero as $x\to\infty$ that is uniform for $t\in(0,T)$.
For example, to obtain uniform \textit{upper} control, observe that
for every $x_1>0$ and $\ep>0$, we have $y(x,t)<\ep$ for all $x\in [0,2x_1]$ 
and sufficiently small $t>0$ by assumption. Therefore we can compare the solution $y$ to 
the shrinking circle starting with the  circle of radius $x_1$ centred at $(x_1,x_1+\ep)$,
which has radius $\sqrt{x_1^2-2t}$,
to establish that $y(x_1,t)\leq x_1+\ep-\sqrt{x_1^2-2t}$ for all $t\in (0,\min\{T,\half x_1^2\})$. 
Taking the limit $\ep\downto 0$ gives the uniform estimate
$$y(x_1,t)\leq x_1-\sqrt{x_1^2-2t}$$
for all $t\in (0,\min\{T,\half x_1^2\})$ and all $x_1>0$,
which implies both the uniform convergence as $t\downto 0$ and the uniform spatial decay.

It follows from Part \ref{W_lower_item_new} of Lemma \ref{RAW_lem} that there exists a universal constant $C<\infty$ 
such that 
$$W(x)\leq C \frac{e^{-\frac{x^2}{4}}}{x^2} \qquad \text{ for } x\geq 1.$$
Since $\w(x,t):=t^{\half}W(xt^{-\half})$, by definition, this implies
$$\w(x,t)\leq Ct^\frac{3}{2} \frac{e^{-\frac{x^2}{4t}}}{x^2} 
\qquad \text{ for } x\geq t^\half,$$
that is, the second inequality of \eqref{yWexp}.
For every $\ep>0$ and sufficiently small $t>0$, depending on $\ep$, we have 
$y(x,t)<\ep$ for $x\geq 0$ by the uniform convergence just established. Therefore we can invoke the comparison principle to establish that $y(x,t)\leq \w(x,t)+\ep$ for all $t\in (0,T)$  and all $x>0$. 
Taking the limit $\ep\downto 0$ gives $y(x,t)\leq \w(x,t)$, and repeating the argument for the solution $-y(x,t)$ then gives the first inequality of \eqref{yWexp}.
\end{proof}

\section{Crossing arguments}
\label{crossing_arg_sect}

In this section we use crossing arguments to prove 
Theorem \ref{ES_improvement_thm}, showing that the size of the gradient $|y_x|$ can be controlled in a sharp way in terms of the height $y$ and the time $t$, and its consequence Corollary \ref{ES_imp_cor}. 
We will also prove our delayed height estimate, Theorem \ref{refine_thm1},
and its consequence Corollary \ref{refine_cor1}.

Angenent \cite{angenent1991parabolic}  developed a principle that the number of crossing points of two solutions to curve shortening flow will be a 
(weakly) decreasing function of time, under reasonable hypotheses.
We use this principle to prove the following.

\begin{lem}
\label{y_W_lem}
Suppose that 
$y:\R\times (0,T)\to (0,\infty)$
is a positive smooth solution to \eqref{GCSF}. 
Then for all 
$t\in (0,T)$
if we define 
\beql{x_t_def}
x_t:=\w(y(0,t),t)=t^\half W\left(y(0,t)t^{-\half}\right)>0,
\eeq
which makes 
$$y(0,t)=\w(x_t,t):=t^{\half}W\left(x_t t^{-\half}\right),$$
then 
$$y(x,t)-\w(x+x_t,t)\leq 0\qquad\text{ for }x\in (-x_t,0),$$
and 
$$y(x,t)-\w(x+x_t,t)\geq 0\qquad\text{ for }x > 0.$$
\end{lem}

\begin{figure}
\centering

\begin{tikzpicture}[scale=.83]


\draw [->] (-3,0) -- (5,0); 
\draw [->] (0,0) -- (0,5) node[right]{$\color{red} \ \w(\cdot,t)$};

\draw[thick, domain=1:5, samples=50, red] 
plot (\x, {
(1/(\x))
});
\draw[thick, domain=1:5, samples=50, red] 
plot ( {(1/(\x))},\x) ;

\draw [thick, dotted] (0.5,2) -- (0.5,0) node[below]{$x_t$};
\draw [thick, dotted] (0.5,2) -- (0,2) node[left]{$y(0,t)$};

\draw [dotted] (2,0.5) -- (0,0.5) node[left]{$x_t$};
\draw [dotted] (2,0.5) -- (2,0) node[below]{$y(0,t)$};

\draw[thick, blue] plot[smooth] coordinates {  (-3,2.5) (-1.5, 3) (-0.8, 3.2) (0, 2)  (2, 0.7)   (3, 1)   (4, 0.5)  (5, 0.4)    };

\path (-3,3) node[above right]{$\color{blue} y(\cdot,t)$};

\end{tikzpicture}
%
%
\begin{tikzpicture}[scale=.83]


\draw [->] (-3,0) -- (5,0); 
\draw [->] (0,0) -- (0,5) node[right]{$\color{red} \ \w(\cdot+x_t,t)$};

\draw[thick, domain=1:5, samples=50, red] 
plot (\x-0.5, {
(1/(\x))
});
\draw[thick, domain=1:5, samples=50, red] 
plot ( {(1/(\x))-0.5},\x) ;

\draw [thick, dotted] (-0.5,5) -- (-0.5,0) node[below]{$-x_t$};

\draw[thick, blue] plot[smooth] coordinates {  (-3,2.5) (-1.5, 3) (-0.8, 3.2) (0, 2)  (2, 0.7)   (3, 1)   (4, 0.5)  (5, 0.4)    };

\path (-3,3) node[above right]{$\color{blue} y(\cdot,t)$};
\draw (2,0) node[below]{$\color{white}y(0,t)$};

\end{tikzpicture}

\caption{Illustration of hypotheses (left) and conclusion (right) of Lemma \ref{y_W_lem}.}
\label{crossing_fig1}
\end{figure}

There are two ways that this property will be applied. 
The first, in the spirit of Angenent's work \cite{angenent1991parabolic}
(see also Ilmanen \cite{ilmanen}, Clutterbuck \cite{clutterbuck_thesis} and 
Nagase-Tonegawa \cite{nagase2009interior}) 
is that it controls the gradient of $y(\cdot,t)$ in terms of the gradient of $W$. The second is that it will give us a lower bound on the area under the graph of $y(\cdot,t)$, depending on $y(0,t)$,  that will effectively give us an upper bound on $y$.

In practice, we will apply this lemma to a shifted flow $(x,t)\mapsto y(x+\ti x,t)$ for some shifting amount $\ti x\in\R$, or to a flipped and shifted flow
$(x,t)\mapsto y(-x+\ti x,t)$. This will give us upper and lower bounds on the gradient of $y(\cdot,t)$ at every point, for example.

\begin{proof}[Proof of Lemma \ref{y_W_lem}]
First, by extending the flow in time, we may assume that $T=\infty$.

The intuitive reason for this lemma to be true is that the graphs of $\w(\cdot+x_t,s)$ and $y(\cdot,s)$ 
for infinitesimal $s>0$ should cross precisely once, so 
the graphs cannot cross more than once at the later time $t>s$ by Angenent's principle \cite{angenent1991parabolic}.
Because $\w(x+x_t,t)\to\infty$ as $x\downto -x_t$, this forces the graph of $\w(\cdot+x_t,t)$ to lie above that of 
$y(\cdot,t)$ for $x<0$ and below for $x>0$.
We now give the rigorous proof.

First we claim that it suffices to prove the result 
in the case that $y$ extends to a smooth flow including $t=0$. That is, we may view $y$ as a restriction of a smooth solution 
$y:\R\times [0,\infty)\to (0,\infty)$.
To obtain the result for a general $y:\R\times (0,\infty)\to (0,\infty)$ at some time $t_0>0$, we could apply the restricted case to the shifted flow $\ti y (x,t):=y(x,t+\ep)$ at time $t=t_0-\ep$, for any $\ep\in (0,t_0)$, and then take a limit $\ep\downto 0$.

Next we claim that it suffices to prove the result 
in the case that $y$ is strictly bounded below by some $\de>0$, i.e., that $y:\R\times [0,\infty)\to (\de,\infty)$.
To obtain the result for a general $y:\R\times [0,\infty)\to (0,\infty)$ at some time $t_0>0$, we could apply the restricted case to the flow 
$\hat y(x,t):=\de+ y(x,t)$, also at time $t=t_0$, 
and then take a limit $\de\downto 0$.

To handle the remaining case that 
$y:\R\times [0,\infty)\to (\de,\infty)$, 
we first make a slight adjustment to $x_t$ by defining,
for $\eta>0$,
$$x_t^\eta:=\w(y(0,t),t+\eta)
=(t+\eta)^\half W\left(\frac{y(0,t)}{(t+\eta)^\half}\right)>0$$
so that
\beql{origin_crossing}
y(0,t)=\w(x_t^\eta,t+\eta).
\eeq
By making $\eta>0$ small enough, which makes 
$x_t^\eta$ approach $x_t$ and makes the part of the graph of
$\w(\cdot + x_t^\eta, \eta)$ lying above height $\de$ approach the vertical line $x=-x_t$, 
we can ensure that
the graph of the initial data $y(\cdot,0)$ 
meets the graph of 
$\w(\cdot + x_t^\eta, \eta)$
at precisely one point, and that the graphs cross transversely there, as illustrated in Figure \ref{Weta_cross}.


\begin{figure}
\centering
\begin{tikzpicture} 

\draw [->] (-3,0) -- (5,0) ; 
\draw [->] (1,0) -- (1,5) ; 

\draw[thick, domain=0.6:5, samples=50] 
plot (\x, {
(0.36/(\x))
});
\draw[thick, domain=0.6:5, samples=50] 
plot ( {(0.36/(\x))},\x);


\draw[dashed] (0,0) node[below]{$ - x^\eta_t$} -- (0, 5) ; 
\draw[dashed] (-3,1.4) -- (5,1.4) node[right]{$\delta$} ;

\draw[thick] plot[smooth] coordinates { (-3, 2)      (-1, 3.3) (0, 2.2) (1, 4.5)  (2, 1.6)   (3, 2)   (4, 2)  (5, 3)    };
\draw (5,3) node[right]{$y(\cdot, 0)$};
\draw (5,.2) node[right]{$\w(\cdot + x_t^\eta, \eta)$} ;

\end{tikzpicture}
\caption{Single crossing between $y(\cdot, 0)$ and $\w(\cdot + x_t^\eta, \eta)$.}
\label{Weta_cross}
\end{figure}

By Angenent's principle \cite{angenent1991parabolic}, this implies that 
the graph of  $y(\cdot,t)$ 
is crossed only once by the graph of 
$\w(\cdot + x_t^\eta, t+\eta)$, and that must be at $x=0$
by \eqref{origin_crossing}.
In particular, we have 
$$y(x,t)-\w(x+x_t^\eta,t+\eta)\leq 0\qquad\text{ for }x\in (-x_t^\eta,0),$$
and 
$$y(x,t)-\w(x+x_t^\eta,t+\eta)\geq 0\qquad\text{ for }x > 0.$$
Taking $\eta\downto 0$ completes the proof.
\end{proof}

\begin{proof}[Proof of Theorem \ref{ES_improvement_thm} and Corollary \ref{ES_imp_cor}]
By scaling, we need only prove the estimate at $t=1$.
By translating the solution in $x$, we may derive the estimate \eqref{HCG_est} at $x=0$.
Lemma \ref{y_W_lem} (keeping in mind that $x_1=W(y(0,1))$)
tells us that
$$y_x\geq W'(W(y))$$
at $(x,t)=(0,1)$.
Similarly, by applying the same lemma to $\ti y(x,t):=y(-x,t)$ tells us that
$$y_x=-\ti y_x \leq -W'(W(y))$$
at $(0,1)$.
Combining gives
$$|y_x|\leq -W'(W(y)).$$
Because $W=W^{-1}$, we have 
$$-W'(W(y))=\frac{1}{-W'(y)},$$
so
\beql{y_x_1_W}
|y_x|\leq \frac{1}{-W'(y)}.
\eeq
But our formula for $W'$ in terms of area from Part \ref{W_prime_area_item} of Lemma \ref{RAW_lem} gives
$$-W'(y)=\tan\ca_0(y),$$
which then yields the inequality of \eqref{HCG_est} of Theorem \ref{ES_improvement_thm}.
The rest of \eqref{HCG_est} follows instantly from 
\eqref{a0a1_juggling}.

To continue to Corollary \ref{ES_imp_cor}, we instead compute
$$-W'(y)\geq \arctan (-W'(y))=\ca_0(y):=\half yW(y)+\si(y)\geq \half yW(y),$$
and so \eqref{y_x_1_W} gives
$$|y_x|\leq \frac{2}{y W(y)}.$$
By Part \ref{W_lower_item_new} of Lemma \ref{RAW_lem}, for sufficiently large $y$, say $y\geq Y$,  we have 
$$W(y)\geq d e^{d^2/4}\frac{e^{-y^2/4}}{y^2}$$
so
$$|y_x|\leq C y e^{y^2/4}$$
for $C$ universal. But if $y<Y$ then \eqref{y_x_1_W} gives us $|y_x|<C$, with $C$ universal. Together, these estimates complete the proof.
\end{proof}

Lemma \ref{y_W_lem} will be used in a different way to prove the delayed height control estimate of Theorem \ref{refine_thm1}.

\begin{proof}[{Proof of Theorem \ref{refine_thm1}}]
By translation of the solution, it suffices to prove the estimate \eqref{delayed_height_est} at $x=0$.
By parabolically scaling the flow and adjusting $A_0$ accordingly it suffices to prove the estimate at $t=1$.
The hypothesis that $t>\frac{A_0}{\pi}$ now becomes 
$A_0<\pi$.

By Lemma \ref{y_W_lem},  we have
$$y(x,1)\geq \w(x+x_1,1)= W(x+x_1)
\qquad\text{ for }x>0,$$
where $x_1=W(y(0,1))$ comes from \eqref{x_t_def}.
In particular, 
\beqa
\int_0^\infty y(x,1)dx 
& \geq \int_0^\infty  W(x+x_1)  dx\\
&= \si(x_1)\\
&= \si(W(y(0,1)))
\eeqa
By repeating the argument for $\ti y(x,t):=y(-x,t)$,
we find that 
$$\int_{-\infty}^0 y(x,1)dx \geq \si(W(y(0,1))),$$
and adding gives 
$$2\si(W(y(0,1)))\leq \|y(\cdot,1)\|_{L^1(\R)}=A_0,$$
where we have used Corollary \ref{est_thm_harnack_cor} to Theorem \ref{est_thm_harnack} in the final equality.
Rearranging gives 
$$y(0,1)\leq W\circ \si^{-1}\left(\frac{A_0}{2}\right)$$
as required.
\end{proof}

\begin{proof}[{Proof of Corollary \ref{refine_cor1}}]
Our task is to obtain asymptotics for the estimate 
\eqref{delayed_height_est} for $t$ just beyond the threshold time, or equivalently for 
$\ep:=\frac{\pi}2-\frac{A_0}{2t}>0$ small.

Let $X>0$ be the (small) value for which the area under the graph of $W$ over $(0,X)$ is $\ep$.
Thus $\si^{-1}\left(\frac{A_0}{2t}\right)=X$. 
If we write $Y:=W(X)$, equivalently $X=W(Y)$, then the upper bound for $y$ given by 
Theorem \ref{refine_thm1} will be $t^\half Y$. 

The rapid decay of $\si$ relative to $W$ that was established in 
Part \ref{tail_item} of Lemma \ref{RAW_lem} tells us that
most of the area $\ep$ under $W$ over $(0,X)$ is made up by the area $XY=YW(Y)$ of the rectangle 
$(0,X)\times (0,Y)$. Precisely, we have  
$$Y W(Y)\leq \ep = Y W(Y)+\si(Y) 
\leq Y W(Y)\left[1+\frac{2}{Y^2}\right]$$
and so 
$$\ep=Y W(Y)\big(1+o(1)\big)$$
as $\ep\downto 0$ (which forces $Y\to\infty$).
By Part \ref{W_lower_item_new} of Lemma \ref{RAW_lem}, this gives
$\frac{c}{Y} e^{-Y^2/4}= \ep (1+o(1))$ for explicit $c:=2de^{d^2/4}$, and so 
$$Y^2/4= \log c - \log Y -\log \ep + o(1).$$ 
Since $Y\to\infty$ as $\ep\downto 0$, 
for small enough $\ep>0$ the $\log Y$ term is large enough to dominate both the $\log c$ term and the error term, and so
$$Y^2/4 \leq -\log \ep,$$
which then implies that 
$$y\leq t^\half Y \leq 2t^\half\sqrt{-\log\ep},$$
for sufficiently small $\ep>0$ as required.
\end{proof}

\section{Example of delayed parabolic regularity}
\label{delay_ex_sect}

In this section we prove Theorem \ref{refined_ex}, and hence also Theorem \ref{baby_uncontrollable_thm}.

We apply Theorem \ref{est_thm_harnack} to each $y_0^n$. A first consequence is that each  resulting flow $y^n$ satisfies $y^n\geq 0$. 
By uniqueness, the solutions $y^n$ 
will remain even, i.e., $y^n(x,t)=y^n(-x,t)$.
Applying Lemma \ref{rapid_decay_cpt_spt_lem} to the flows 
$(x,t)\mapsto y^n(x+\frac{1}{n},t)$
gives us upper bounds
\beql{W_bd_right}
\textstyle y^n(x,t)\leq \w(x-\frac{1}{n},t),
\eeq
for $x>\frac{1}{n}$,
and by evenness of $y^n$, we have the symmetric estimate
\beql{W_bd_left}
\textstyle y^n(x,t)\leq \w(-(x+\frac{1}{n}),t),
\eeq
for $x<-\frac{1}{n}$.
For each $t>0$, the function $x\mapsto y^n(x,t)$ will remain increasing for $x<0$ and thus decreasing for $x>0$. One way of seeing this is that for every $c>0$, 
the initial data $y_0^n$ will cross the horizontal line $\{y=c\}$ precisely twice for $n>c$, and this number cannot increase during the 
flow by Angenent's crossing principle \cite{angenent1991parabolic}.

One consequence of \eqref{W_bd_left}, recalling the definition of $\w$ given in \eqref{W_self_sim}, is that for every $x<-\frac{1}{n}$, we can control $\ca(x,t)$ from Theorem \ref{est_thm_harnack}.
Indeed, we have 
\beqa
\ca(x,t) &:=\int_{-\infty}^x y^n(s,t)ds
\leq 
t^{\half} \int_{-\infty}^x W\left(-(s+{\textstyle \frac{1}{n}})t^{-\half}\right) ds\\
&= t \int_{-(x+\frac{1}{n})t^{-\half}}^\infty
W(r)dr\\
&=t\si\left(-(x+{\textstyle\frac{1}{n}})t^{-\half}\right),
\eeqa
and so \eqref{second_arctan_est} gives
\beqa
\arctan y_x^n(x,t) &\leq \half\si\left(-(x+{\textstyle\frac{1}{n}})t^{-\half}\right)+\frac{\pi}4\\
&\leq \half\si(\de)+\frac{\pi}4 < \frac{\pi}{2}
\eeqa
throughout the region $(x,t)$ for which 
$-(x+\frac{1}{n})t^{-\half}\geq \de>0$.
Because this region will swallow up the space-time subset 
$(-\infty,x_0]\times [0,T]$ (arbitrary $x_0<0$ and $T>0$)
once $\de>0$ is small enough and $n$ large enough, we obtain gradient bounds for $y^n$ on $(-\infty,x_0]\times [0,T]$ for sufficiently large $n$, and so parabolic regularity and compactness allow us to pass to a subsequence to extract a smooth local limit $y^\infty$ on 
$(-\infty,0)\times [0,\infty)$.
By evenness of $y^n$, we also obtain local convergence on 
$(0,\infty)\times [0,\infty)$.

We now establish similar local convergence on 
$\R\times (\tau,\infty)$.
For this it is most efficient to apply the consequence of the Harnack estimate given in Remark \ref{even_harnack_rmk}, with $A_0=1$.
That tells us that for $t>\tau:=\frac{1}{\pi}$ and $x<0$ we have
$$\textstyle 0\leq y_x^n \leq \tan\left[\frac14(\frac{1}{t}+\pi)\right],$$
and by evenness this gives that for all $x\in\R$ and $t>\tau=\frac{1}{\pi}$ we have
$$\textstyle
|y_x^n|(x,t) \leq \tan\left[\frac14(\frac{1}{t}+\pi)\right]<\infty.$$
Having this  upper bound on $|y_x^n|$ over
$\R\times [\tau+\de,\infty)$, any $\de>0$, that is uniform in $n$, allows
us to invoke parabolic regularity theory to get uniform local bounds on the $C^k$ norms of $y^n$.
In particular, we can pass to a further subsequence to obtain a smooth local limit $y^\infty$ on 
$\R\times (\tau,\infty)$ as required. 

The remaining claim of the theorem is the exact formula for $y^\infty$ for $t\in (0,\tau)$. By evenness we only need consider $x>0$.
By passing \eqref{W_bd_right} to the limit $n\to\infty$ 
we find that
$$y^\infty(x,t)\leq \w(x,t),$$ 
for all $x>0$ and all $t>0$.
We need only to prove the reverse inequality for $t\in (0,\tau)$.
To do this, we first need a lower bound for $y^n(0,t)$.
By \eqref{W_bd_right}, keeping in mind Part \ref{P3} of Lemma \ref{RAW_lem}, we have
$$\|y^n\|_{L^1([\frac1{n}, \infty))}
\leq \frac{\pi t}2.$$
By evenness, and the fact that $\|y^n\|_{L^1(\R)}=1$, 
we have
$$\|y^n\|_{L^1([0,\frac1{n}])}\geq\half -\frac{\pi t}2.$$
Because $y^n$ achieves its maximum at $x=0$, this then implies
$$y^n(0,t)\geq \frac{n}{2}\left(
1-\pi t\right),$$
so $y^n(0,t)\to\infty$ as $n\to\infty$, locally uniformly in 
$t\in [0,\tau)=[0,\frac1\pi)$.

As a consequence of this lower bound at $x=0$, we can use the right-angled wedge solution as a \textit{lower} barrier also. More precisely, for every $t_0\in (0,\tau)$
and every $\ep>0$,  we can use the shifted solution 
$(x,t)\mapsto -\ep+\w(x+\ep,t)$
as   a lower bound for $y^n$ over $[0,\infty)\times [0,t_0]$ for sufficiently large $n$.
Taking a limit first $n\to\infty$, and then $\ep\downto 0$, we find that
$$y^\infty(x,t)\geq \w(x,t),$$ 
for all $x>0$ and $t\in (0,t_0)$.
Since $t_0\in (0,\tau)$ is arbitrary, we obtain the lower bound for all $t\in (0,\tau)$ as required.

This completes the proof of Theorem \ref{refined_ex}.

\section{\texorpdfstring{$L^p$ initial data for $p>1$}{Lp initial data for p>1}}
\label{Lpsect}

We prove Theorem \ref{Lp_reg} closely following the strategy in \cite{TY1}.

\begin{proof}
Without loss of generality we may assume that $y_0\geq 0$, and hence also $y\geq 0$ and prove just the desired upper bound for $y$.
This is because otherwise we can replace $y_0$ by $|y_0|$ (which has the same $L^p$ norm) and then observe that by Corollary \ref{comp_princ_loc_lip} the Ecker-Huisken flow starting with $|y_0|$ must dominate both $y$ and $-y$.

Consider now the $p$, $y$ and $t$ from Theorem \ref{Lp_reg}.
If $t\geq \|y_0\|_{L^1}\geq \frac{2}{\pi}\|y_0\|_{L^1}$ then Corollary \ref{refine_cor}  gives
$\sup y(\cdot,t)\leq Ct^\half $,
but by assumption, 
\beql{t_half}
t^\half = t^{-\frac{1}{p-1}}t^{\frac{p+1}{2(p-1)}}\leq 
t^{-\frac{1}{p-1}} \|y_0\|_{L^p}^{(\frac{p+1}{2(p-1)})(\frac{2p}{p+1})}
=t^{-\frac{1}{p-1}} \|y_0\|_{L^p}^{\frac{p}{p-1}}.
\eeq
so
$$\sup y(\cdot,t)\leq C t^{-\frac{1}{p-1}} \|y_0\|_{L^p}^{\frac{p}{p-1}}$$
as required.
If instead 
$t< \|y_0\|_{L^1}$ 
then we apply Corollary \ref{refine_cor} to the flow $\ti y$ starting with initial data $(y_0-k)_+$ for  $k\geq 0$ chosen so that 
$\|(y_0-k)_+\|_{L^1}= t$ 
to give $\sup \ti y(\cdot,t)\leq Ct^\half$.
The comparison principle of Corollary \ref{comp_princ_loc_lip} ensures that 
$$y(x,t)\leq k+\ti y(x,t) $$ 
since the inequality holds at $t=0$, so we need to bound $k$. 
By estimating 
$$\|y_0\|_{L^p}^p\geq \int_{\{y_0\geq k\}}y_0^p
\geq k^{p-1}\int_{\{y_0\geq k\}}y_0\geq k^{p-1}\|(y_0-k)_+\|_{L^1}=k^{p-1} t, $$ 
we obtain
$$k\leq \|y_0\|_{L^p}^\frac{p}{p-1}
t^{-\frac{1}{p-1}},$$
and combining gives
$$y(x,t)\leq  t^{-\frac{1}{p-1}}\|y_0\|_{L^p}^\frac{p}{p-1} + Ct^\half,$$
for universal $C$.
Appealing to \eqref{t_half} again gives the conclusion.
\end{proof}

\section{Flowing from rough initial data}
\label{L1_exist_sect}

In this section we prove the existence of solutions to graphical curve shortening flow starting with $L^1$ initial data, as asserted in  Theorem \ref{L1_exist_thm}.
The proof will be broken down into the existence of Lemma \ref{y_construct_lem} and the attainment of initial data of Lemma \ref{init_data_attained_lem}.

We begin by taking a sequence of mollifications $y^n_0$ of $y_0$.
The smooth functions $y^n_0$ converge to $y_0$ in $L^1(\R)$ as $n\to\infty$.
Let $y^n:\R\times [0,\infty)\to \R$ be the  Ecker-Huisken flows having 
$y^n_0$ as initial data. We will use our new estimates to get uniform control on these flows.

\begin{lem}[Construction of solution]
\label{y_construct_lem}
The functions $y^n$ and all their derivatives are locally bounded on $\R\times (0,\infty)$, independently of $n$.
In particular, after passing to a subsequence in $n$, the solutions $y^n$ converge smoothly locally on $\R\times (0,\infty)$ to a solution
$y:\R\times (0,\infty)\to\R$ of \eqref{GCSF}.
\end{lem}

\begin{proof}
It suffices to prove $C^k$ estimates for $y^n$ on time intervals $[3\ep,L]$, for 
arbitrary $0<3\ep<L$. The estimates may depend on $\ep$, $L$ and $k$, but must be uniform in $n$.
By parabolic regularity theory it  suffices  to prove that we have uniform (independent of $n$) bounds on $|y^n|$ and the first derivatives $|y^n_x|$ over the time interval $[2\ep,L]$. 
This gives from time $2\ep$ to time $3\ep$ for parabolic regularity to produce a $C^k$ bound.

By Theorem \ref{ES_improvement_thm}, or Corollary \ref{ES_imp_cor}, it suffices to 
prove that we have a uniform $C^0$ estimate  for $y^n$ on time intervals  $[\ep,L]$.
Note that if we obtain a $C^0$ estimate $|y^n|\leq M$ for $t\in[\ep,L]$, then we would apply Theorem \ref{ES_improvement_thm} or Corollary \ref{ES_imp_cor}
to the 
\textit{positive} solution $M+1+y^n$, which is bounded above by $2M+1$,
starting from time $\ep$. From time $2\ep$ (say) onwards we would obtain a uniform 
bound on $|y^n_x|$.
It suffices to 
prove that we have a uniform ($n$-independent) upper bound  for $y^n$ at time $\ep$ alone, since that bound will persist, by the comparison principle, 
and we obtain a lower bound for free by applying our arguments to $-y^n$.


We will obtain the uniform upper bound for $y^n$ at time $\ep$ using our delayed height estimate of Theorem \ref{refine_thm1}, in the form of Corollary \ref{refine_cor}.
First note that the  mollification of $y_0$ (i.e.  $y_0^n$) 
is dominated by the mollification of $(y_0)_+$, which in turn has the same $L^1$ norm
as $(y_0)_+$. Therefore 
\beql{moll_fact}
\|(y^n_0)_+\|_{L^1(\R)}\leq 
\|(y_0)_+\|_{L^1(\R)}.
\eeq

If $\ep\geq\frac{2}{\pi}\|(y_0)_+\|_{L^1(\R)}$, 
we  apply Corollary \ref{refine_cor}
not to the flow $y^n$ starting with $y^n_0$, but to the Ecker-Huisken solution $\ti y^n$ starting with $(y^n_0)_+$.
We apply it at time $\ep\geq\frac{2}{\pi}A_0$ where $A_0=\|(y^n_0)_+\|_{L^1(\R)}$ for this application.
The conclusion is an upper bound
$$\ti y^n(\cdot, \ep) \leq K:=C\ep^\half,$$ 
for universal $C$.
%
By the comparison principle (comparing $\ti y^n$ to $y^n$ over the time interval $[0,\ep]$ using Corollary \ref{comp_princ_loc_lip}) we find that 
$$y^n(\cdot, \ep) \leq \ti y^n(\cdot, \ep)\leq K,$$
which provides the desired $n$-independent upper bound for $y^n(\cdot, \ep)$.

If instead $\ep<\frac{2}{\pi}\|(y_0)_+\|_{L^1(\R)}$, then we pick $k>0$ such that 
$$\ep=\frac{2}{\pi}\|(y_0-k)_+\|_{L^1(\R)},$$
and note that $k$ depends only on $\ep$ and $y_0$.
We can then repeat the argument above with $\ti y^n$ now the Ecker-Huisken solution starting with $(y_0^n-k)_+$, with $A_0=\|(y^n_0-k)_+\|_{L^1(\R)}$, and applied now
at time $\ep\geq \frac{2}{\pi}A_0$ for this different $A_0$.
The conclusion is again that 
$$\ti y^n(\cdot, \ep) \leq K,$$
for $K$ independent of $n$.
This time we complete the argument by comparing $y^n$ to $k+\ti y^n$ to give
$$y^n(\cdot, \ep) \leq  k+\ti y^n(\cdot, \ep)\leq k+K,$$
which again provides the desired $n$-independent upper bound for $y^n(\cdot, \ep)$.
\end{proof}

We have constructed a candidate solution $y$. To complete the proof of Theorem \ref{L1_exist_thm}
we have to show that 
it attains the correct initial data, i.e. we can interchange the limits $n\to\infty$ and $t\downto 0$.

\begin{lem}[Initial data attained]
\label{init_data_attained_lem}
The solution $y:\R\times (0,\infty)\to\R$ constructed in Lemma \ref{y_construct_lem} satisfies $y(\cdot,t)\to y_0$ in $L^1(\R)$ as $t\downto 0$.
\end{lem}

Before we can prove this, we need to understand how quickly two nearby solutions can separate under the flow. We abbreviate $y(\cdot,t)$ by $y(t)$.

\begin{lem}[Separation estimate]
\label{separation_lem}
Let $y^1$, $y^2$ be two smooth solutions to \eqref{GCSF} on $I\times [0,T)$, where $I\subset\R$ is any open interval, finite or infinite.
Then 
$$\|\big(y^1(t)-y^2(t)\big)_+\|_{L^1(I)}\leq \|\big(y^1(s)-y^2(s)\big)_+\|_{L^1(I)}+2\pi (t-s)$$
for all $0\leq s\leq t<T$.
\end{lem}

\begin{proof}
Let $\vph\in C_c^\infty(I,[0,1])$. 
It suffices to prove that 
\beql{suff_to_p}
\|\big(y^1(t)-y^2(t)\big)_+\vph\|_{L^1(\R)}\leq \|\big(y^1(s)-y^2(s)\big)_+\vph\|_{L^1(\R)}
+\pi (t-s)\int|\vph_x| dx
\eeq
for all $0\leq s\leq t<T$ because we could then take a sequence of  cut-off functions $\vph$ that increase to $\chi_I$, each having $\int|\vph_x|dx=2$.
Moreover, we may assume that $s>0$ since otherwise we can apply the result with $s>0$ and take a limit $s\downto 0$.

The function $(y^1-y^2)\vph$ is smooth on $I\times[s,t]$ and so
$(y^1-y^2)_+\vph \in W^{1,1}(I\times(s,t))$ with weak derivative
$$\pl{}{t}(y^1-y^2)_+\vph=\left\{
\begin{aligned}
(y^1_t-y^2_t)\vph & \qquad\text{if }y^1>y^2\\
0 & \qquad\text{otherwise.}
\end{aligned}
\right.
$$
(See, e.g., \cite[Lemma 7.6]{GT}.)
Therefore the function $t\mapsto \|\big(y^1(t)-y^2(t)\big)_+\vph\|_{L^1(\R)}$
lies in $W^{1,1}((s,t))$ and 
\beqa
\frac{d}{d t}\|\big(y^1(t)-y^2(t)\big)_+\vph\|_{L^1(\R)}
& = \int_{\{y^1>y^2\}}(y^1_t-y^2_t)\vph\,dx\\
& = \int_{\{y^1>y^2\}}\big(\arctan(y^1_x)-\arctan(y^2_x)\big)_x\vph\,dx
\eeqa
where the integrals are over the time slice at time $t$.
Let $R>0$ be large enough so that the support of $\vph$ lies within $(-R,R)$.
The set $\{y^1>y^2\}\intersect (-R,R)$ is a countable  union of disjoint open intervals
$I_i=(a_i,b_i)$.
Moreover, by definition of $I_i$ we have $y^1_x\geq y^2_x$ at $a_i$ if $a_i\neq -R$ and 
$y^1_x\leq y^2_x$ at $b_i$ if $b_i\neq R$. Because $\arctan$ is an increasing function, 
this then implies that $\arctan(y^1_x)\geq \arctan(y^2_x)$ at $a_i$ if $a_i\neq -R$ and
$\arctan(y^1_x)\leq \arctan(y^2_x)$ at $b_i$ if $b_i\neq R$.
Integrating by parts on the time slice at time $t$, keeping in mind that $\vph(-R)=\vph(R)=0$, gives
\beqa
\int_{I_i}\big(\arctan(y^1_x)-\arctan(y^2_x)\big)_x\vph\,dx
&= -\int_{I_i}\big(\arctan(y^1_x)-\arctan(y^2_x)\big)\vph_x\,dx\\
& \qquad+\left[\big(\arctan(y^1_x)-\arctan(y^2_x)\big)\vph\right]_{x=a_i}^{x=b_i}\\
&\leq -\int_{I_i}\big(\arctan(y^1_x)-\arctan(y^2_x)\big)\vph_x\,dx.
\eeqa
Therefore
\beqa
\frac{d}{d t}\|\big(y^1(t)-y^2(t)\big)_+\vph\|_{L^1(\R)} & \leq -\int_{\{y^1>y^2\}}\big(\arctan(y^1_x)-\arctan(y^2_x)\big)\vph_x\,dx\\
& \leq \pi\int |\vph_x|\,dx,
\eeqa
because $|\arctan|$ is bounded by $\frac\pi{2}$.
This can be integrated to give \eqref{suff_to_p}.
\end{proof}

\begin{proof}[{Proof of Lemma \ref{init_data_attained_lem}}]
We need to show that for any $\eta>0$ we have
\beql{objective}
\|y(t)-y_0\|_{L^1(\R)}\leq \eta
\eeq
for sufficiently small $t>0$. 
Because mollifications converge in $L^1$, for some fixed $m\in\N$ and all $n\geq m$ we have
\beql{ing2}
\|y_0^n-y_0\|_{L^1(\R)}\leq \eta/16.
\eeq
In particular,  we have
\beql{mini_ing}
\|y_0^n-y_0^m\|_{L^1(\R)}\leq \eta/8.
\eeq
We next claim that $y^m(t)\to y^m_0$ in $L^1(\R)$ as $t\downto 0$,
and in particular
\beql{ing1}
\|y^m(t)- y^m_0\|_{L^1(\R)}\leq \eta/4
\eeq
for sufficiently small $t>0$.
The mollification $y^m_0$ is globally Lipschitz, so the estimates of Ecker-Huisken 
\cite{EH1} tell us that $y^m(t)\to y^m_0$ (globally) uniformly as $t\downto 0$. 
The claim will follow if we can demonstrate that 
$\|y^m(t)\|_{L^1(\R)}\to \|y^m_0\|_{L^1(\R)}$ as $t\downto 0$, or
(equivalently in this case) that 
\beql{limsup_req}
\limsup_{t\downto 0}\|y^m(t)\|_{L^1(\R)}\leq \|y^m_0\|_{L^1(\R)},
\eeq
i.e., that 
$y^m(t)$ does not lose $L^1$ norm at spatial infinity as $t\downto 0$. To prove 
this we apply the separation estimate of Lemma \ref{separation_lem} in the case that $I=\R$, $s=0$, $y^2\equiv 0$ and $y^1=y^m$. 
That tells us that 
$$\|\big(y^m(t)\big)_+\|_{L^1(\R)}\leq \|\big(y^m_0\big)_+\|_{L^1(\R)}+2\pi t$$
for all $t\geq 0$. Repeating with $y^m$ replaced by $-y^m$ and adding then gives
$$\|y^m(t)\|_{L^1(\R)}\leq \|y^m_0\|_{L^1(\R)}+4\pi t$$
for all $t\geq 0$, which gives \eqref{limsup_req} and hence
the claim and \eqref{ing1}.

We now appeal to the separation estimate of Lemma \ref{separation_lem} in the case that $I=\R$, $s=0$, $y^1=y^n$ and $y^2=y^m$. 
Together with \eqref{mini_ing} this implies that
\beqa
\|\big(y^n(t)-y^m(t)\big)_+\|_{L^1(\R)} &\leq \|\big(y^n_0-y^m_0\big)_+\|_{L^1(\R)}+2\pi t\\
& \leq \eta/8+2\pi t\\
&\leq \eta/4
\eeqa
for 
$t\in (0,\frac{\eta}{16\pi}]$.
By repeating with $n$ and $m$ switched, we obtain the corresponding estimate
$\|\big(y^m(t)-y^n(t)\big)_+\|_{L^1(\R)}\leq \eta/4$, and so
$$\|y^n(t)-y^m(t)\|_{L^1(\R)} \leq \eta/2$$
for $t\in (0,\frac{\eta}{16\pi}]$.
We would now like to take the limit $n\to\infty$. We know from the construction of $y(t)$ in Lemma \ref{y_construct_lem} that $y^n\to y$ smoothly locally on $\R\times (0,\infty)$, and in particular $y^n(t)\to y(t)$ locally uniformly on $\R$. 
Since 
\beq
\|y^n(t)-y^m(t)\|_{L^1([-R,R])}\leq \|y^n(t)-y^m(t)\|_{L^1(\R)}\leq \eta/2
\eeq
for any $R>0$ and any $t\in (0,\frac{\eta}{16\pi}]$, we may take first the limit 
$n\to\infty$ and then $R\to\infty$ to conclude that
$$\|y(t)-y^m(t)\|_{L^1(\R)}\leq \eta/2$$ 
for $t\in (0,\frac{\eta}{16\pi}]$.
%
%
Combining with \eqref{ing1} and \eqref{ing2} (the latter with $n=m$)
gives \eqref{objective} for sufficiently small $t>0$, as required.
\end{proof}

It turns out to be possible to localise the key estimates in the paper, specifically the Harnack inequality in Section \ref{basic_harnack_sect},
in order to improve the $L^1$ existence of this paper to more general initial data such as 
a nonatomic Radon measure of possibly infinite mass, although the initial data will necessarily be attained in a weaker sense than in Theorem \ref{L1_exist_thm}. 
The following theorem is proved in \cite{arjun_thesis, alt_arj_thesis}.

\begin{thm}[{Flowing from measure initial data}]
\label{meas_exist_thm}
Let $\nu  = y_0\cl^1 + \nu_\mathrm{sing} $
be a {\normalfont nonatomic} real-valued Radon measure\footnote{By real-valued Radon measure we mean an element of the dual to $C_c^0(\R,\R)$, cf. \cite[\S 1.8, Theorem 1]{evans_gariepy}.}, 
$y_0\in L^1_{loc}(\R)$,
decomposed into its absolutely continuous and singular parts, and let $\Omega = \R \setminus 
\mathrm{supp}(\nu_\mathrm{sing})$.
Then there exists a smooth solution 
$y:\R\times (0,\infty)\to \R$ to 
\eqref{GCSF} 
such that 
\beqa
\label{last_one}
y(\cdot,t) \cl^1 \rightharpoonup \nu & \qquad\text{weakly on }\R  \\
y(\cdot,t) \to y_0 & \qquad\text{strongly in } \, L^1_{loc}(\Omega),
\eeqa
as $t\downto 0$.
\end{thm}

The notion of convergence of measures taken in 
\eqref{last_one} is weak-$*$ convergence, i.e., for all $\vph\in C_c^0(\R)$ we must have
$$\int y(\cdot,t)\vph dx\to\int \vph d\nu$$
as $t\downto 0$.

\emph{Acknowledgements:} 
AS was supported by EPSRC studentship EP/R513374/1.
PT  was supported by EPSRC grant EP/T019824/1.
For the purpose of open access, the authors have applied a Creative Commons Attribution (CC BY) licence to any author accepted manuscript version arising.

\vskip 0.2cm

\noindent
AS: 
\url{https://sites.google.com/view/sobnack}


\noindent
{\sc Department of Mathematics, The University of Texas at Austin, TX 78712, USA.}

\noindent
PT: 
\url{https://warwick.ac.uk/fac/sci/maths/people/staff/peter_topping}

\noindent
{\sc Mathematics Institute, University of Warwick, Coventry,
CV4 7AL, UK.}

\end{document}